\documentclass[11pt]{amsart}

\usepackage[left=3cm,tmargin=3cm,centering]{geometry}

\usepackage{xcolor}
\usepackage[backref=page]{hyperref}
\definecolor{couleur_cite}{rgb}{0.05,.4,0.05}
\definecolor{couleur_link}{rgb}{0.05,0.05,0.4}
\hypersetup{
bookmarksopen,bookmarksnumbered,
colorlinks,
linkcolor=couleur_link,citecolor=couleur_cite,
}

\usepackage[initials]{amsrefs}  
\usepackage{bm,amsmath,amssymb,amsthm,amscd}

\usepackage[all]{xy}
\usepackage{stmaryrd}

\usepackage{amsfonts}

\usepackage{abbreviations}

\usepackage[latin1]{inputenc}

\numberwithin{equation}{section}
\usepackage{longtable}    
\allowdisplaybreaks[1]    

\theoremstyle{plain}
\newtheorem{theorem}{Theorem}
\newtheorem*{theoremA}{Theorem A}

\newtheorem{corollary}{Corollary}[section]

\newtheorem{proposition}{Proposition}[section]
\newtheorem*{proposition*}{Proposition}
\newtheorem{lemma}{Lemma}[section]
\newtheorem*{lemma*}{Lemma}
\newtheorem*{corollary*}{Corollary}

\theoremstyle{definition}
\newtheorem{definition}{Definition}[section]

\newtheorem{remark}{Remark}

\theoremstyle{remark}

\numberwithin{equation}{section}

\newcommand{\sumchi}{\frac{1}{h(D)}
\sum_{\chi\in\widehat{\MCl}_D}}
\newcommand{\Lprime}{L'(1/2,f\times\chi)}
\newcommand{\itemb}{\item[$\bullet$]}

\begin{document}

\title[]{A non-split sum of coefficients of modular forms} 

\author[]{Nicolas Templier}   
\address{Institute for Advanced Study, School of Mathematics, 08540 Princeton, NJ, USA}
\email{nicolas.templier@normalesup.org}

\date{\today}

\begin{abstract}
We shall introduce and study certain truncated sums of Hecke eigenvalues of $GL_2$-automorphic forms along quadratic polynomials. A power saving estimate  is established and new applications to moments of critical $L$-values associated to quadratic fields are derived. An application to the asymptotic behavior of the height of Heegner points and singular moduli is discussed in details.
\end{abstract}

\maketitle


\tableofcontents


\section{Introduction.} 
Upper bounds for sums of arithmetic functions is a classical and central problem in analytic number theory. In this paper we shall introduce certain sums of coefficients of modular forms that may be used as variants of shifted convolution sums in certain circumstances.

\subsection{Main result} Let $\FmH=\{x+iy,\quad y>0 \}$ be the Poincar\'e upper-half plane. Let $f:\FmH \rightarrow \BmC$ be a classical modular form of weight $2$, trivial Nebentypus and odd squarefree level. Let
\begin{equation}\label{intro:Fourier}
 f(z)=\sum^\infty_{n=1} n^{1/2} \lambda_f(n) e^{2i\pi nz},\quad \forall z\in\FmH
\end{equation}
be its normalized Fourier expansion at infinity. We shall establish the following estimate:
\begin{theorem}\label{th:main} There are absolute constants $\eta,\eta'>0$ such that the bound
\begin{equation}\label{eq:th:main}
\sum_{N<n<2N}\lambda_f(n^2+d)\ll_{f} N^{1-\eta},
\end{equation}
holds uniformly for all couples $(d,N)$ where $d$ is a prime number with $d\equiv 3 \pmod{4}$ and $N$ is a positive number with $d^{1/2-\eta'}\le N \le d^{1/2+\eta'}$.
\end{theorem}

\remark The left-hand-side is a sum of length $N$. The direct application of Deligne's bound $|\lambda_f(n)|\le \tau(n)$, where $\tau$ is the divisor function, would yield the majoration $\ll N\log N$. The bound~\eqref{eq:th:main} saves a small power of $N$.

The typical example is when $N=d^{1/2}\to \infty$. In the theorem we allow some more freedom for $N$ because this flexibility is needed for applications and comes naturally from the method of proof.
 
\remark The exponents $\eta,\eta'$ could be made explicit and are equal to the $\frac{\eta_4}{1+2s}$ given in section~\ref{sec:proof:smooth} (we shall assume for simplicity $\eta=\eta'$ in the sequel). Our approach is not well-suited to optimize the value of the exponents because it relies on a large number of transformations, each one carrying waste.

\remark Independently, V.~Blomer~\cite{Blom08} has established a result similar to~\eqref{eq:th:main} when $d$ is fixed and $N\to \infty$. Here the constraint $N\le d^{1/2+\eta'}$ makes the length of the $n$-sum shorter.

\remark In the present article we do not work out the case $d<0$. However let us recall the known case where $d$ were the opposite of a perfect square, $d=-h^2$ say. Then the quadratic polynomial $n\mapsto n^2-h^2$ would split and the left-hand side of~\eqref{eq:th:main} would essentially reduce to 
\begin{equation}\label{eq:classical-shifted}
 \sum_{N<n<2N} \lambda_f(n-h)\lambda_f(n+h)
\end{equation}
(because of the multiplicativity of $\lambda_f$). A.~Selberg~\cite{Selb56} was the first to study these sums. Producing a non-trivial estimate for~\eqref{eq:classical-shifted} is the Shifted Convolution Problem (SCP) for two $GL(2)$ forms, whose resolution is a cornerstone for many further developments~\footnote{In the classical SCP we may choose $N$ as small as $h^\theta$ where $\theta$ is the exponent towards Ramanujan-Petersson, which is to be compared with the assumption $d^{1/2-\eta'}\le N$ in Theorem~\ref{th:main}.} (see~\cite{cong:park:mich} for a good survey). This distinction between split and non-split polynomials justifies why we may call the left-hand side of~\eqref{eq:th:main} \Lquote{a non-split sum}.

\remark The first occurrence of a non-split quadratic polynomial in this kind of problem appears in a work of C.~Hooley \cite{Hool63}. The result of that paper and further developments, notably~\cite{DFI95}, have had an important influence to the present paper. We refer to a forthcoming survey for a detailed discussion of the nexus; a key insight is that a consequence of Duke's Theorem~\cite{Duke88} is the uniform distribution of 
\begin{equation}
 \{\dfrac{\nu}{q}:\quad \nu^2+d\equiv 0 \pmod{q},\ \nu \in \BmZ/q\BmZ,\ 1\le q\le d^{1/2} \}
\end{equation}
inside $\BmR/\BmZ$ as $d\rightarrow +\infty$. This fact is not used explicitly in the proof of Theorem~\ref{th:main}, but nevertheless lies in the background and has provided a guideline through our work.

\remark Although we did not state it explicitly, the proof of Theorem~\ref{th:main} is valid for modular forms $f$ of arbitrary even weight $2k$ and odd squarefree level\footnote{Note that we do not claim any precise bound in the weight nor level aspect. In this article all the constants involved in the bounds $\ll_f$ are polynomials (with a large exponent) in the weight and the level of $f$.}. The only change is that the Bessel function $J_1$ is replaced by the Bessel function $J_k$. The proof also works for Maass forms of odd squarefree level, although it yields statement~\eqref{eq:th:main} in its \emph{smooth version only} (slightly weaker) because Deligne's bound is not available for Maass forms.\footnote{By \Lquote{smooth version} we mean that $\sum_{N<n<2N}$ is replaced by $\sum_n V(n/N)$ where $V$ is smooth ($\CmC^\infty$) of compact support.}

The condition that the level of $f$ is odd and squarefree is a technical difficulty that simplifies the computations in sections~\ref{sec:sums} and \ref{sec:circle}. We expect that a variant of Theorem~\ref{th:main} would hold for all cuspidal automorphic forms on $GL(2)_\BmQ$.

\remark Theorem~\ref{th:MAIN} from section~\ref{sec:proof} provides a slightly more general version. The difference with Theorem~\ref{th:main} is on the restriction that $d$ is a prime number. In Theorem~\ref{th:MAIN}, we allow $d$ to be squarefree with all its prime factors $>d^\epsilon$, where $\epsilon>0$ is fixed in advance. This assumption on $d$ is a technical assumption that arises in the explicit computations from section~\ref{sec:sums}. We expect that estimate~\eqref{eq:th:main} would hold for all positive integers $d$, see also the next remark.

\remark In Theorem~\ref{th:MAIN}, we also allow a square part, replacing $d$ by $de^2$ with $e\ge 1$, because it is needed for applications. The dependence on $e$ is polynomial: $e^{O(1)}$. It should be possible to obtain a sharp estimate in this parameter. This would involve a fine analysis at the finite places and should be closely related to a recent theorem of V.~Vatsal~\cite{Vats02}. We shall not discuss this interesting issue in the present paper.

\remark In a recent work of R.~Holowinski~\cite{Holo08}, which relies on very different methods (sieve and partial results towards Sato-Tate), estimates that save a power of $\log N$ in several SCP of \emph{absolute values} of Hecke eigenvalues are established. It would be interesting to investigate bounds for
\begin{equation}
 \sum_{N <n< 2N} \abs{\lambda_f(n^2+d)}
\end{equation}
(for instance with $d$ fixed and without the constraint $N\le d^{1/2+\eta}$ in a first attempt). When $-d$ is not a perfect square it is not clear how one could proceed.

\subsection{Moments of \protect{$L$}-functions.}\label{sec:intro:cases} Theorem~\ref{th:main} arises in the study of moments of $L$-functions associated to quadratic number fields. In this section we recall what is already known and in the next one we explain our new applications. Let $D<0$ be the discriminant of an imaginary quadratic field $K=\BmQ(\sqrt{D})$. Let $\CmO_D$ be the ring of integers and $\MCl_D$ the ideal class group. One may associate to unitary characters $\chi\in \widehat{\MCl}_D$ on this group many interesting $L$-functions.\footnote{In the sequel we always choose the {\em unitary normalization} for the $L$-series of principal automorphic forms $\pi$: the functional equation links $L(s,\pi)$ with $L(1-s,\pi)$, in particular the \emph{critical line} is $\MRe s=\Mdemi$.} It is important and challenging to determine asymptotically the average of the critical values of these $L$-functions. The average is with respect to $\chi\in\widehat{\MCl}_D$ (one speaks of the \emph{moments of the family} in the classical terminology introduced by~\cite{book:KS}). Main examples are as follows:
\begin{itemize}
\item[(A)] The Hecke $L$-function $L(s,\chi)$ is the most organic. The first and second moment of $L(1/2,\chi)$ have been studied by Duke, J.~Friedlander and H.~Iwaniec \cites{DFI4,Temp:Eisenstein}. Quantitative non-vanishing has been obtained by V.~Blomer~\cite{Blom04}. A subconvex bound in the $D$-aspect has been established in \cite{DFI8}.

\item[(B)] Let $\psi$ be a \Lquote{canonical} Hecke character on $\BmQ(\sqrt{D})$ of conductor $\sqrt{D}\CmO_D$ (the terminology is from \cite{Rohr80:canonical}). Consider the Hecke $L$-functions $L(s,\psi\chi)$, and assume that the sign of the functional equation is $+1$. Quantitative nonvanishing of $L(1/2,\psi\chi)$ has been studied by D.~Rohrlich and others~\cites{Rohr80:galois,Rohr80:nonvanishing,MR82,RY99,Yang99,MY00,Masr07:quantitative}. The asymptotic for the first moment has been computed by C.~Liu, L.~Xu, B~Kim, R.~Masri and T.~Yang~\cites{LX04,MY07,Masr07:asymptotics,MY07}. A subconvex bound in the $D$-aspect follows from~\cite{DFI8}.

\item[(C)] Let $f$ be a primitive modular form or a primitive Maass form. The $L$-series $L(s,f\times \chi)$ may be defined via the Rankin-Selberg method. A subconvex bound in the $D$-aspect has been established in~\cites{Mich04,HM06}. The sign of the functional equation is $\pm 1$. When the sign is $+1$, the first moment of $L(1/2,f\times \chi)$ and the quantitative nonvanishing have been obtained by Ph.~Michel and A.~Venkatesh~\cite{MV05}.

\item[(D)] Let $L(s,f\times \chi)$ be as in (C), but assume that the sign of the functional equation is $-1$ and $f$ is holomorphic of weight $2$. Partial results on the first moment of the special derivative $L'(1/2,f\times \chi)$ have been obtained by G.~Ricotta and T.~Vidick~\cites{RV05,RT08} (on average over $D$) and by Michel and Venkatesh~\cite{MV05} (under an unproven hypothesis), and by the author~\cites{Temp:these,Temp:height} (a lower bound for the first moment).
\end{itemize}
In all four Cases (A-D) the conductor of the $L$-function is $\sim D^2$ (in Case (A) the second moment is the most relevant and the conductor of $L(s,\chi)^2$ is $D^2$), the size of the family is $h(D)$ the class number (which is roughly $|D|^{1/2}$ as $D\to -\infty$). The respective moments thus are:
\begin{equation}\label{intro:moments}
\begin{aligned}
 \frac{1}{h(D)}&\sum_{\chi\in \widehat{\MCl}_D} \abs{L(1/2,\chi)}^2 ;&
 \frac{1}{h(D)}&\sum_{\chi\in \widehat{\MCl}_D} L(1/2,\psi\chi);\\
 \frac{1}{h(D)}&\sum_{\chi\in \widehat{\MCl}_D} L(1/2,f\times \chi); &
 \frac{1}{h(D)}&\sum_{\chi\in \widehat{\MCl}_D} L'(1/2,f\times \chi).
\end{aligned}
\end{equation}

In these four Cases, \Lquote{period formulas} have been extensively studied. These formulas link each $L$-value (or derivative in Case (D)) to a certain period of a quadratic cycle on a Shimura curve. As a corollary each $L$-value is nonnegative as predicted by the GRH. For the convenience of the reader, we briefly locate these period formulas in the literature. Case (A) is \emph{Hecke's formula}, see \cite{book:sieg:adva}. The formula for Case (B) is due to F.~Rodriguez-Villegas and D.~Zagier \cites{Rodr91,Rodr93,RZ93} when the root number is $+1$ and to Yang \cites{Yang00} when the root number is $-1$. When $\chi$ trivial, Case (C) is due to J.-L.~ Waldspurger~\cite{Wald85b}. When $\chi$ is arbitrary and $f$ holomorphic, it is due to B.~Gross and Zagier~\cites{Gros87,GZ}, see also~\cites{MW07,Ichi07}. When $\chi$ is arbitrary and $f$ is a Maass form the formula is due to S.-W.~Zhang~\cite{Zhan01b} (see also A.~Popa~\cite{Popa06} for real quadratic fields). Case (D) is \emph{the Gross-Zagier formula}~\cite{GZ} which has been recently generalized by Zhang and X.~Yuan and W.~Zhang~\cites{Zhan01b,cong:heeg04:zhan,YZZ08}.

These period formulas yield a closed expression for the moments~\eqref{intro:moments} above. Michel and Venkatesh observed~\cite{MV05}, by analogy with Vatsal's work~\cite{Vats02}, that these expressions can be combined with Duke's theorem to determine the asymptotics of the moments. In~\cite{MV05} they address Case (C). Then Case (B) has been treated in~\cites{Masr07:quantitative,Masr07:asymptotics,MY07} and Case (A) in~\cite{Temp:Eisenstein}. 

The Case (D) is more subtle because the period formula involves \emph{heights} of Heegner points. The article~\cite{Temp:height} provides a short argument that yields a lower bound which is sufficient for certain applications. A more ambitious approach that would yield the \emph{exact asymptotic} with power saving for these heights has been developed in the author's PhD thesis~\cite{Temp:these} following ideas from \cite{cong:ICM06:MV}*{section~2.4}. This approach contains several difficulties that are not yet surmounted.\footnote{Except when $f$ is the level $11$ form, where we observed~\cite{Temp:these}*{section 6.4} that huge cancellations occur in the regularized local heights explicited by Gross-Zagier.}

\subsection{An application of Theorem~\ref{th:main}.}
In order to solve Case (D) completely, we shall forget about these deep period formulas alluded to in the previous section and go back to pure analytic methods that make use of the functional equation only.\footnote{I am very grateful to Peter Sarnak who suggested me to do so} Although we do not make it explicit this approach could settle also the Case (C) in a uniform manner\footnote{In Case (A) see also~\cite{DFI95}, in Case (B) see also~\cite{MY00}}. In some sense the estimate~\eqref{eq:th:main} from Theorem~\ref{th:main} should be considered as lying in the heart of the question of moments of $L$-functions associated to class group characters, as long as the conductor is $D^2$.

\begin{theorem}\label{cor:moment}
Let $f$ be a weight $2$ primitive modular form of odd squarefree level $N$. There exists an absolute constant $\eta_5>0$ such that the following estimate holds uniformly on the prime discriminants $D$ satisfying $\chi_D(N)=1$,
\begin{multline}\label{eq:cor:moment}
\frac{1}{h(D)}
\sum_{\chi\in\widehat{\MCl}_D}
L'(1/2,f\times\chi)=
4\frac{L^{(N)}(1,\chi_D)}{\zeta^{(N)}(2)}L(1,\MSym^2
f)\biggl[\Mdemi\log |DN| +
  \frac{L'^{(N)}}{L^{(N)}}(1,\chi_D)+\\
\quad+\frac{L'}{L}(1,\MSym^2 f)-\frac{\zeta'^{(N)}}{\zeta^{(N)}}(2)-\gamma-\log 2\pi +O_f(|D|^{-\eta_5})\biggr].
\end{multline}
Here $\gamma$ is Euler constant ; $L(\cdot,\MSym^2 f)$ is the symmetric square $L$-function ; the superscripts in $\zeta^{(N)}$ and $L^{(N)}$ indicate that the Euler factors at primes divisors of $N$ have been removed.
\end{theorem}

As consequence of the Gross-Zagier formula we may deduce very precise informations on the height of Heegner points on elliptic curves. This is explained in section~\ref{sec:analytic}.

\remark The asymptotic behavior of $\Mdemi\log\abs{D}+\frac{L'}{L}(1,\chi_D)$ is recalled in \S~\ref{sec:intro:LD}. As a consequence, the brackets in the right-hand side of~\eqref{eq:cor:moment} tends to $+\infty$ as $D$ gets large, which is consistent with the fact that the left hand-side is nonnegative \emph{for every} $D$, as follows from the Gross-Zagier formula (or would follow from the GRH).

\begin{remark}
The residual quantity $\frac{L'}{L}(1,\MSym^2 f)-\gamma -2\pi$ appears in other contexts related to height functions, in particular for the self-intersection of the dualizing sheaf of $X_0(N)$, see~\cite{AU97}. This is not a coincidence.
\end{remark}

\remark The fact one can bypass the use of period formulas in the proof of Theorem~\ref{cor:moment} has a significance and may be exploited further to gain deep insights: a common ingredient, explicit or implicit, in all the methods (analytic and geometric ones) is a \emph{relative trace formula} for the arithmetic pair $GL_2(\BmQ) \supset \BmQ(\sqrt{D})^\times$. The real difference between the geometric and the analytic approach lies in the order in which the steps are performed. Hopefully there should exist a unifying framework which comprises both period formulas and asymptotics for moments of critical values of $L$-functions. We do not develop the idea further in this paper. See also~\cites{MW07,RR05}.

\subsection{Outline of the proof of Theorem~\ref{cor:moment}}
The first task is to express the special value in a convenient fashion. This is done by applying the approximate functional equation method, see identity~\eqref{approx}. This method has been used several times in the past and is quite robust since it relies only on the functional equation see. For instance it puts Case (C) and (D) on equal footing.

Then it is possible to extract a main term, this is discussed in \S~\ref{sec:cor:main} by means of the counting function $r_D$, see~\eqref{def-rD} and~\eqref{def-rD-dag}. The remainder term contains a combination of sums of $\lambda_f$ against quadratic polynomials and Theorem~\ref{th:main} is exactly what we need to save a small power of $\abs{D}$, see \S~\ref{sec:cor:remain}.

\subsection{Outline of the proof of Theorem~\ref{th:main}.} 
First of all we need to stress out that our proof relies on an auxiliary result, Theorem~A whose proof will be given elsewhere~\cite{Temp:quadratic} because it involves quite different techniques. The present paper provides all the detailed steps from Theorem~A to Theorem~\ref{th:main}. Main ideas underlying a slightly longer proof of Theorem~\ref{th:main}, including Theorem~A, have been outlined in~\cite{Temp:cras}. 

The first step, carried out in section \ref{sec:circle}, is to solve analytically $(n^2+d=m$) via the $\delta$-symbol method~\cite{DFI2}. The structure of the argumentation is close to~\cite{Pitt}. Roughly speaking the effect of the $\delta$-symbol method is to replace the Fourier coefficients $\lambda_f(m)$ by sums of Kloosterman sums. An important difference with previous applications of the $\delta$-symbol is that we are concerned with savings in the \emph{sums over the moduli} and not only in the square-root cancellations of complete exponential sums. Also the choice of certain parameters is slightly different.

The next step is to apply Poisson summation formula, see \S~\ref{sec:proof:poisson}. Then a peculiar kind of complete exponential sum shows up, see~\eqref{exp:intro}. It may be viewed as a generalization of Sali\'e sums and carries a square-root cancellation. This cancellation is sufficient to recover the naive bound $N^{1+\epsilon}$ in Theorem~\ref{th:main}.

The final saving is included in the sum over the moduli $q$. This is the object of section~\ref{sec:sums}. First we observe that the exponential sum is related to Jacobi forms. Then we quote without proof an estimate (Theorem~A) which contains the desired saving. This estimate ultimately follows from Iwaniec's celebrated bound~\cite{Iwan87}.

\subsection{Chowla-Selberg versus Gross-Zagier.}
To our knowledge this is the first time a link between these two popular period formulas is stated. Our results imply that when the discriminant of the quadratic field is large, the Chowla-Selberg and Gross-Zagier formulas become very close to each other.\footnote{In~\cite{KRY04}, S.~Kudla, M.~Rappoport and Yang discuss a distinct situation which involves derivatives of Eisenstein series as a generating series for the heights. In a recent preprint, J.~Bruinier and Yang~\cite{BY08} consider yet another situation; a difference with our discussion is that they consider the~\emph{trace} of the Heegner points, which corresponds to choosing $\chi=\Mun$ in the Gross-Zagier formula~\eqref{GZ}.}

This may be visualized by the diagram of \Lquote{equalities} below. Each equality has to be understood up to an explicit multiplicative constant. The error terms and the multiplicative constants are discussed at several places throughout the text, the diagram portrays the \emph{formal aspect}. The \emph{main term} in Theorem~\ref{cor:moment} may thus be interpreted in a beautiful way:
\begin{equation}
\begin{CD}
 \frac{\abs{D}^{1/2}}{h(D)^2}
 \sum\limits_{\chi\in \widehat{\MCl}_D}
 \Lprime
 @=
 \frac{L'}{L}(1,\chi_D)+\Mdemi\log \abs{D}+O(1)\\
  @| @| \\
  \widehat{h}(\varphi(z_D))
 @=
 h\subt{Fal}(E_D) +O(1)
\end{CD}
\end{equation}
Explanation: the top row is purely analytic in nature (Theorem~\ref{cor:moment}), and very common in the theory of moments of $L$-functions: the moment at $1/2$ of a family is asymptotic to a special value at $1$ of $L$-functions on groups of smaller rank. The second row is closely related to a key result by Faltings that compares the Faltings height and Weil height functions on moduli spaces, up to logarithmic terms\footnote{The Faltings comparison (Proposition~\ref{prop:falt}) would yield only a $O(\log \log \abs{D})$ instead of $O(1)$ at the bottom right. But it turns out that for the special case of Heegner points this may be improved as the Theorem~\ref{cor:moment} shows.}. The first column is the Gross-Zagier formula. The second column is the Chowla-Selberg formula.

\subsection{Asymptotic height of singular moduli.} Let $j_D$ be the $j$-invariant of an elliptic curve CM by $\CmO_D$. The theory of complex multiplication says that is is an algebraic integer unique up to Galois conjugation. In explicit terms, one may choose $j_D=j(\dfrac{1+i\sqrt{\abs{D}}}{2})$, where
\begin{equation}
 j(z)=\frac{1}{q}+744+196884q+\cdots
 ,\quad
 q=e^{2i\pi z}
\end{equation}
is the classical $j$-function: $\FmH\rightarrow \BmC$.

The literature is very prolific on the arithmetic of CM-elliptic curves (see for instance the references listed in~\cite{BJO06} for the theoretic aspect and listed in~\cite{Brok08} for the algorithmic aspect), but an answer to the following simple and natural question does not seem to exist\footnote{Quoting~\cite{BJO06}*{p.378}: \Lquote{these polynomials are generally quite complicated and the basic problem of computing them and their roots has long history}. This is the only answer one usually may read.}. \emph{What is the behavior of the na\"ive height}
\begin{equation}
 h(j_D),\quad
 \text{as $D\to -\infty$}?
\end{equation}
Since this question is partly related to Theorem~\ref{cor:moment}, we take the opportunity to answer it in section~\ref{sec:geom} (see Proposition~\ref{prop:naive}) by a geometrical approach, recalling several known facts on periods of CM-elliptic curves. Although this question is perhaps known to experts, we believe it is important to have a place that discusses it for the sake of non-experts (like the author).

This question is very natural because the na\"ive height measures the arithmetic complexity of an algebraic number. The singular moduli $j_D$ are algebraic integers and it is clear from many sources that its complexity grows quickly with the discriminant $D$. Here are some evidences that are related to $h(j_D)$.

In~\cite{GZ85}*{Table 1} the factorization of the absolute norm of $j_D$ is displayed. The explicit formula for this norm proved by Gross-Zagier implies the nice result that the prime factors are all less than $\abs{D}$. Let $P_D$ be the minimal polynomial over $\BmZ$ of $j_D$. It is of degree $h(D)$ and sometimes called \Lquote{class polynomial} because $\BmQ(\sqrt{D},j_D)$ is the Hilbert class field $H_D$ of $\BmQ(\sqrt{D})$.  For example~\cite{YZ97} displays\footnote{It is further observed that the polynomials $P_D$ \Lquote{have coefficients
of astronomical size even for quite modest discriminants $D$}, and the authors introduce and compute a variant called Weber polynomials that have far smaller coefficients and still generate the Hilbert class field. However from the point of view of heights both $P_D$ and the Weber polynomials have, \emph{up to a multiplicative constant}, nearby asymptotic complexity. One may understand why the Weber polynomials are of smaller size, especially for small values of $D$, by contemplating the leading exponent $q^{-\frac{1}{48}}$ in the Fourier expansion of the Weber function which is to be compared with the $q^{-1}$ for the $j$-function.} the value of $P_{-55}$. A standard inequality for heights yields (see~\cite{book:BG}*{Proposition 1.6.6}):
\begin{equation}
 \sum_{\sigma \in \MCl_D} \log^+ \abs{j^\sigma_D}=h(D)h(j_D)=\log M(P_D) 
\end{equation}
Here $M(\cdot)$ denotes the Mahler measure of the polynomial. It is clear (see also~\cite{book:BG}*{Proposition 1.6.6}) that the latter quantity is larger than:
\begin{equation}
 \ge \log \abs{P_D(0)}
 =\Mdemi \log \abs{\TmN_{H_D/\BmQ} j_D}
 =\sum_{\sigma \in \MCl_D} \log \abs{j^\sigma_D}.
\end{equation}

Actually a simple application of Duke's theorem yields
\footnote{\emph{sketch of proof.} One needs to control the $j_D^\sigma$ whose norm are close to $0$. Since $\BmP^1\simeq X(1)$ and we may view $X(1)(\BmC)$ as the hyperbolic quotient $SL_2(\BmZ)\SB \FmH$, this is the same as controlling how close the Heegner points of discriminant $D$ may be to $\rho=e^{i\pi/3}=\dfrac{1+i\sqrt{3}}{2}$. But the logarithmic distance is at least $\log D$ as one may deduce quickly from the explicit representation $\dfrac{b+\sqrt{D}}{2a}$ of Heegner points. And Duke's theorem states that $\{j^\sigma_D\}$ are equidistributed for the hyperbolic measure. This is enough to conclude that the negative contribution $\sum_{\sigma \in \MCl_D} \log^{-} \abs{j^\sigma_D}$ is $o(h(D)\log \abs{D})$, which is what we need. The error term is obviously poor since one had to isolate a small region around $\rho$ and to apply Duke's theorem afterwards.}:
\begin{equation}
 \Mdemi \log \abs{\TmN_{H_D/\BmQ} j_D} \sim h(D)h(j_D),
 \qtext{as $D\to -\infty$.}
\end{equation}
It is possible to run a similar argument for the asymptotic of $\log \abs{\TmN_{H_D/\BmQ} (j_D-1728)}=\log \abs{P_D(1728)}$ for which an exact prime factorization is also displayed in~\cite{GZ85}. We leave the details to the interested reader.

Another interesting quantity is the discriminant of $P_D$, which is directly related to the index $I_D$ of $\BmZ[j_D]$ in its integral closure. At least when $D$ is a prime discriminant, one has (the absolute discriminant of $\BmQ(j_D)$ when $D$ is prime is computed in the book~\cite{book:Gross:LN}):
\begin{equation}
 \Mdisc(P_D)=I^2_D \abs{D}^{\frac{h(D)-1}{2}}.
\end{equation}
The value of $I_D$ is displayed in~\cite{GZ85}*{Table 1}, and computed in~\cite{GZ85}*{Corollary 4.8}. From~\cite{book:BG}*{Proposition 1.6.9} one has the rough bound:
\begin{equation}
 \frac{1}{h_D}\log \Mdisc(P_D) \le (2h(D)-2)h(j_D) + \log h(D).
\end{equation}
It would be interesting, but perhaps difficult, to obtain a good lower bound for $I_D$ as $D\to -\infty$. The results in~\cite{GZ85} seem to indicate that the growth of $I_D$ is indeed very fast.

\subsection{Log-derivative at $1$ of Dirichlet $L$-series.}\label{sec:intro:LD}
It is convenient to introduce the following notation for a quantity that will appear often in the text:
\begin{equation}
 \CmL_D:=\Mdemi\log \abs{D}+
 \frac{L'}{L}(1,\chi_D).
\end{equation}
In this paragraph we recall the asymptotic behavior of this quantity. The Riemann Hypothesis for $L(s,\chi_D)$ would imply $\frac{L'}{L}(1,\chi_D)=O(\log \log \abs{D})$, so that\footnote{In~\cite{MM00} it is proven unconditionally that $\limsup\limits_{D\to -\infty}
\frac{L'(1,\chi_D)}{L(1,\chi_D)\log \log \abs{D}}\ge 1/2$ and $\liminf\limits_{D\to -\infty}
\frac{L'(1,\chi_D)}{L(1,\chi_D)\log \log \abs{D}}\le -1/2$. This tends to show that these quantities are indeed delicate} one expects $\CmL_D \sim \Mdemi \log \abs{D}$.

Unconditionally it is possible to prove:
\begin{equation}
  (\frac{1}{4}-\epsilon)\log |D| \le \CmL_D \ll_{\epsilon} \abs{D}^\epsilon
\end{equation}
for any $\epsilon>0$ and $D$ large enough. The upper bound follows from Siegel theorem\footnote{and it is very difficult to improve it unconditionally. As explained for instance in~\cite{MM00}*{Theorem 4.2} such an improvement would be intimately related with the absence of Siegel zeros}. The lower bound is a standard consequence of Burgess estimate (see section 3 of~\cite{Temp:Eisenstein} for a proof). 

\remark In~\cite{Colm98}, P.~Colmez proves the lower bound $\log |D| \ll \CmL_D$. This follows from a uniform version of Weyl's law (proposition~5 in~\cite{Colm98}) which is very classical in analytic number theory (see, e.g., \cite{book:IK04}*{Theorem~5.8}).

\subsection{Notation and convention.} For notational simplicity we shall prove the estimate with $\eta=\eta'$. We shall label the successive exponents arising in the sequel in the following manner:
\begin{equation}
  0 <\eta_5<\eta<\eta_4 <\eta_3 <\eta_2 <\eta_1.
\end{equation}
The exponent $\eta_1$ arises in Theorem~A. Then $\eta_2$ will be chosen sufficiently small compared to $\eta_1$ and so on. The exponent $\eta$ is the one from Theorem~\ref{th:main}. We did not compute its precise value. The exponent $\eta_5$ appears in the proof of Proposition~\ref{prop:rem} and thus in Theorem~\ref{cor:moment}. It is more customary in analytic number theory to keep these choices implicit, but we believe this labelling improves the clarity.

For the height functions, we adopt the conventions from~\cite{book:BG}. If $L$ is an ample divisor, $h_L$ denotes the composition of the naive height with the map to projective space induced by $L$. If the underlying variety is abelian, $\widehat{h}_L$ denotes the canonical height. The $O()$, $o()$, $\sim$, $\ll$ and $\gg$ have their traditional meaning. 

\subsection{Structure of the paper.} The proof of the main Theorem~\ref{th:main} is performed in \S~\ref{sec:proof}. The \S~\ref{sec:sums} contains the estimate on sums of exponential sums while the \S~\ref{sec:circle} builds the variant of the circle method. 
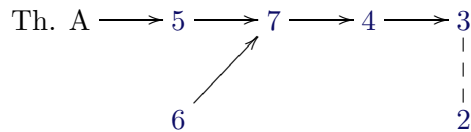
\begin{equation*}
 \xymatrix{%
 \text{Th. A} \ar[r] &
 \ref{sec:sums} \ar[r] &
 \ref{sec:proof} \ar[r] &
 \ref{sec:cor} \ar[r] &
 \ref{sec:analytic}\\
 &
 \ref{sec:circle} \ar[ur] &
 &
 &
 \ref{sec:geom} \ar@{--}[u]
 }
\end{equation*}
The proof of Theorem~\ref{cor:moment} is performed in~\S~\ref{sec:cor}. The application to height of Heegner points is exposed in \S~\ref{sec:analytic}. This is to be compared with a geometric approach in \S~\ref{sec:geom}.

\subsection{Acknowledgments.} The article is partly based on Chapter 12 of the author's PhD thesis~\cite{Temp:these} and some of the results have been announced in~\cite{Temp:cras}.  My indebtness goes to my advisor Philippe Michel for his constant support. I thank Peter Sarnak for insisting on developping the approximate functional equation method for the present family: at that time (May 2007) it was not clear that an estimate like~\eqref{eq:th:main} would exist. I also want to express my gratitude to Philippe Michel and Akshay Venkatesh for letting me search on these problems although they already had distinct interesting ideas (see \cite{MV05}, \cite{cong:ICM06:MV}*{\S~2.4}). I thank Gergely Harcos for introducing me to some of the subtleties of the Shifted Convolution Problem.
My final thank goes to the book~\cite{book:BG}.

\section{Heights of Heegner points -- geometric approach.}\label{sec:geom}
Before proceeding in detail with the proofs of Theorems~\ref{th:main} and \ref{cor:moment}, we discuss a geometric proof of a \emph{weak version} of Theorem~\ref{cor:moment}. The techniques of this section are in a very different flavor than the rest of the text and the reader interested solely on $L$-functions may skip this section. We believe this section will be useful for the reader to gain a better understanding of the objects underlying the moments of quadratic $L$-functions.

Let $E$ be a rational elliptic curve. Let $N$ be its conductor and $\varphi:X_0(N)\rightarrow E$ be a Weil parametrization which exists by Wiles celebrated theorem. Let $\widehat{h}:E(\overline{\BmQ})\rightarrow \BmR_+$ be the N\'eron-Tate height. Let $D$ be a fundamental negative discriminant such that the \emph{Heegner condition} is satisfied: all prime factors of $N$ are split in $\BmQ(\sqrt{D})$. We choose one Heegner point $z_D $ of discrimant $D$ on $X_0(N)$.

The quantity $\dfrac{\widehat{h}(\varphi(z_D))}{\Mdeg(\varphi)}$ is an
arithmetic invariant of the couple $(E,D)$ formed by an elliptic curve $E/\BmQ$ and a compatible discriminant $D$. Actually it depends only on the isogeny class of $E$. We are interested in its behavior as $D$ gets large.

In~\cite{Temp:height}, we established $\liminf\limits_{D\to -\infty} \widehat{h}(\varphi(z_D))>0$ by an equidistribution argument that works in a fairly general situation. We observed also~\cite{Temp:height}*{\S 4} that in the present case of modular curves $X_0(N)$, it is possible to use the geometry of the cusps via rough comparison arguments and established:
\begin{equation}
 \widehat{h}(\varphi(z_D)) \gg_E \CmL_D.
\end{equation}
In the next proposition we shall refine this last result. The proof of the proposition occupies \S~\ref{sec:geom:CS} to \S~\ref{sec:geom:proof}. We may view the present section as a complement of section~4 from \cite{Temp:height}. Let $g(N)$ be the genus of $X_0(N)$ and $\nu(N):=[SL_2(\BmZ):\Gamma_0(N)]$.
\begin{proposition}\label{prop:geom} Let notations and assumptions be as above. Then:
 \begin{equation}\label{eq:prop:geom}
  \frac{\widehat{h}(\varphi(z_D))}{\deg \varphi} 
  \sim 
  \frac{6}{\nu(N)} \CmL_D,
  \qtext{as $D\to -\infty$.}
 \end{equation}
\end{proposition}

\remark It seems difficult to have a good control on the quality of the asymptotic~\eqref{eq:prop:geom} from the geometric approach. This mainly comes from the Proposition~\ref{prop:image} which does not give an explicit error term but merely the existence of a limit. Also Faltings approximation result contains a $\log \log \abs{D}$ in the remaining term which is difficult to remove.

\medskip

During the proof we shall establish\footnote{An explicit formula for Heegner points on Shimura curves \emph{of full level} is the main purpose of~\cite{KRY04}}:
\begin{lemma}\label{lem:height-zD} Let $\widehat{h}:X_0(N)(\overline{\BmQ})\rightarrow \BmR_+$ be as in~\cite{GZ}, see also \S~\ref{sec:analytic:alpha}.
 \begin{equation}
  \widehat{h}(z_D) \sim \frac{6g(N)}{\nu(N)}\CmL_D,
  \qtext{as $D\to -\infty$.}
 \end{equation}
\end{lemma}

\remark One may decompose $\widehat{h}$ on $J_0(N)$ as a sum of the N\'eron-Tate heights on its simple abelian quotients. From this fact one may deduce Lemma~\ref{lem:height-zD} from the analog of Proposition~\ref{prop:geom} for modular abelian variety (which is also proved in the next section~\ref{sec:analytic} by analytic methods).

However this decomposition itself is useless in the proof of Proposition~\ref{prop:geom}. For instance a divisor on $J_0(N)$ may project to zero or to a torsion point on $E$. In the argument below we use the fact that the Heegner points really belong to the \emph{curve} $X_0(N)$ inside $J_0(N)$. Precisely, we make use of Proposition~\ref{prop:image} which automatically removes this possibility (at least for points of large height).

\remark The arguments provided below may be compared with section 4 from~\cite{Temp:height} in a fairly precise way. Although stated in a different language, both proofs are in the same flavor. The \S~\ref{sec:intro:LD} discusses~\cite{Temp:height}*{Lemma 5}. The Proposition~\ref{prop:CS} below covers~\cite{Temp:height}*{Lemma 6}. The Proposition~\ref{prop:falt} covers~\cite{Temp:height}*{inequality (24)}. The Propositions~\ref{prop:naive} and~\ref{prop:image} cover~\cite{Temp:height}*{inequalities (25-28)}.

\subsection{A formula by Chowla and Selberg.}\label{sec:geom:CS} In the early 80's and 90's, articles have been written on the periods of CM elliptic curves (and more generally of CM abelian varieties). In this paragraph we briefly recall the formula we shall need.

Recall that there is a notion of Faltings height of an abelian variety defined over $\overline{\BmQ}$, see e.g.~\cite{cong:arith84:falt2}. Let $E_D$ be an elliptic curve over $\overline{\BmQ}$ with CM by $\CmO_D$. We refer the reader to the book~\cite{book:Gross:LN} for a discussion of the arithmetic properties of these curves. The connexion to periods of CM elliptic curves (and abelian varieties) was first observed by P.~Deligne.

\begin{proposition}[Chowla-Selberg]\label{prop:CS} The Faltings height of $E_D$ depends only on $D$ and is equal to:
\begin{equation}
 2 h\subt{Fal}(E_D)= \CmL_D +c,
\end{equation}
where $c$ is an absolute constant\footnote{we do not display its exact value because it depends on the chosen normalization of $h\subt{Fal}$ which varies from an article to another.}.
\end{proposition}
A proof is to combine Kronecker limit formula for Eisenstein series on $SL_2(\BmZ)\SB \FmH$ and the Hecke period formula. The reader is referred to~\cite{Colm93} or~\cite{KRY04}*{Proposition 10.10} for further discussions around that formula.

\subsection{Approximation of the Faltings height.} In his proof of finiteness theorems for abelian varieties, Faltings~\cite{cong:arith84:falt2}*{\S 3} shows that, up to logarithmic terms, the Faltings height is a multiple of the height of the abelian variety on the moduli space (with respect to an embedding to projective space which is defined in a canonical way).\footnote{this construction is better viewed in the language of metrized line bundles for Arakelov geometry} 

For elliptic curves, one may find a nearby discussion of this fact in~\cite{cong:arith84:silv2}*{Proposition~2.1}. Note that our definition of the Faltings height differs from~\cite{cong:arith84:silv2}.
\begin{proposition}\label{prop:falt} Let $E$ be a semistable elliptic curve defined over $\overline{\BmQ}$ of $j$-invariant $j_E$. Then (the constants are absolute):
\begin{equation}
 O(1)\le h(j_E)-12 h\subt{Fal}(E) \le 6\log (1+h(j_E)) +O(1).
\end{equation}
\end{proposition}

\subsection{Asymptotic height of singular moduli.} The $j$-invariant of $E_D$ which we have denoted $j_D$ is unique up to Galois conjugation. From Propositions~\ref{prop:CS} and \ref{prop:falt} and \S~\ref{sec:intro:LD} we deduce:
\begin{proposition}\label{prop:naive} 
The naive height of $j_D$ satisfies the following asymptotic:
\begin{equation}
 h(j_D) \sim 6 \CmL_D
 ,\quad
 \text{as } D\to -\infty.
\end{equation}
\end{proposition}

\subsection{Image of points of large height.}
Let's recall the following, see~\cite{book:HS00}*{proposition B.3.5}:
\begin{proposition}\label{prop:image}
Let $X$ be a smooth projective curve defined over~$\BmQ$. Let $A,B$ be divisors on $X$ with $\deg(A)\ge 1$. Then:
\begin{equation}
 \lim_{\substack{
 P\in X(\overline{\BmQ})\\
 h_A(P)\to \infty
 }}
 \frac{h_B(P)}{h_A(P)}=\frac{\deg B}{\deg A}.
\end{equation}
\end{proposition}

\subsection{Explicit degree of certain divisors.}\label{sec:geom:degree}
Consider the map $\iota:X_0(N)\rightarrow J_0(N)$ from the modular curve to its Jacobian which sends the cusp $i\infty$ to the origin. Denote by $\pi:X_0(N)\rightarrow X(1)\simeq \BmP^1$ the standard projection which is of degree $\nu(N)=[SL_2(\BmZ):\Gamma_0(N)]$. All the morphisms in the following diagram are defined over $\BmQ$, which enables to consider the image of the Heegner points $z_D$:
\begin{equation}
 \xymatrix{%
 \BmP^1 & \ar[l]_\pi X_0(N)\ar[dr]^\varphi \ar@{^{(}->}[r]^\iota & J_0(N) \ar@{.>}[d]\\
 && E 
 }
\end{equation}
We shall need the precise value of the degree of certain divisors on $X_0(N)$ that are pullbacks by the above maps.
\begin{lemma}\label{lem:degrees}
Let $\Theta$ be the theta divisor on $J_0(N)$ and $\Xi:=\Theta+[-1]^*\Theta$; let $O:=\CmO(O)$ be the line bundle associated to the origin $(O)$ of $E$. Then:
\begin{itemize}
 \item[(i)] $\deg \varphi^* O = \deg \varphi$;
 \item[(ii)] $\deg \iota^* \Xi = 2g(N)$ (the pullback is in the sense of line bundles or divisor classes);
 \item[(iii)] $\deg \pi^* \CmO(1)=\deg \pi=\nu(N)$.
\end{itemize}
\end{lemma}
\begin{proof}
 (i) and (iii) are obvious. Assertion (ii) is classical, see section~8.10 from~\cite{book:BG} for instance.
\end{proof}

\subsection{Proof of Proposition~\ref{prop:geom}.}\label{sec:geom:proof}
Recall that $z_D\in X_0(N)(\overline{\BmQ})$ and that the point $\pi(z_D)$ corresponds to an elliptic curve with CM by $\CmO_D$ hence is conjugate to $j_D$.

We begin by the asymptotic of $\widehat{h}(z_D)$, where we recall that the height $\widehat{h}$ on $X_0(N)$ is such that $\widehat{h} \circ \iota=\widehat{h}_\Xi$. By the formalism of height functions (see, e.g,\cite{book:HS00}*{B.3.2}), one has:
\begin{equation}
 \widehat{h}(z_D)=
 \widehat{h}_\Xi \circ \iota (z_D)
 =h_\Xi \circ \iota (z_D)+O(1)
 =h_{\iota^* \Xi}(z_D)+O(1),
\end{equation}
and:
\begin{equation}
 h(j_D)=h\circ \pi (z_D)=h_{\pi^* \CmO(1)}(z_D) +O(1).
\end{equation}
From Proposition~\ref{prop:naive} and the lower bounds from \S \ref{sec:intro:LD} we deduce $h_{\pi^* \CmO(1)}(z_D)\to \infty$. Since the degree of $\pi^* \CmO(1)$ is positive, we may apply Proposition~\ref{prop:image} which yields:
\begin{equation}\label{pf:height-zD}
 \widehat{h}(z_D)\sim 
 \frac{\deg \iota^*\Xi}{\deg \pi^* \CmO(1)}
 h(j_D)
 ,\qtext{as $D\to -\infty$.}
\end{equation}

By definition of the N\'eron-Tate height we have $\widehat{h}=h_{O}+O(1)$ on $E(\overline{\BmQ})$. Hence:
\begin{equation}
 \widehat{h}(\varphi(z_D))
 =
 h_{\varphi^* O}(z_D) +O(1).
\end{equation}

Since the degree of $\iota^*\Xi$ is positive, $h_{\iota^* \Xi}(z_D)\to \infty$ by the previous result. We may again apply Proposition~\ref{prop:image} which yields by~\eqref{pf:height-zD}:
\begin{equation}
  \widehat{h}(\varphi(z_D))\sim
  \frac{\deg \varphi^* O}{\deg \pi^* \CmO(1)}
 h(j_D).
\end{equation}

Making use of Lemma~\ref{lem:degrees} and Proposition~\ref{prop:naive}, we conclude the proof of the Proposition~\ref{prop:geom}.

\section{Heights of Heegner points -- analytic approach.}\label{sec:analytic}
By the Gross-Zagier formula, the quantity $L'(1/2,f\times \chi)$ is proportional to the N\'eron-Tate height of $\varphi(z_D)$ introduced in the previous section. Theorem~\ref{cor:moment} then yields precise informations about these heights. Although it is possible to carry out the study in greater generality, we stick to the initial Gross-Zagier context~\cite{GZ} which we now proceed to recall.

\subsection{The Gross-Zagier formula.} Assume that the Fourier coefficients of $f$ are rational and let $E$ be the rational elliptic curve associated to $f$ by the Shimura-Taniyama construction and $\varphi:X_0(N)\rightarrow E$. As in the previous section we assume that the \emph{Heegner condition} is satisfied which implies $\chi_D(N)=1$, thus an odd functional equation for $L(s,f\times \chi)$.

The Gross-Zagier formula~\cite{GZ}*{\S I.6} yields:
\begin{equation}\label{GZ}
\sumchi \Lprime
=
\alpha L(1,\chi_D) \frac{\widehat{h}(\varphi(z_D))}{\deg(\varphi)}.
\end{equation}

By combining the formulas given in \cite{GZ}*{pp. 230, 308, 310}, one has the following\footnote{In \S~\ref{sec:analytic:alpha} we give further details on this equality} (see also~\cite{RV05} or~\cite{RT08}*{Remarque 5}):
\begin{equation}\label{alpha}
 \alpha=\frac{2N}{\pi^2}L(1,\MSym^2 f).
\end{equation}

\subsection{Refined asymptotic for the height.}
 Now we may explain the arithmetic significance of the moment in Case (D) (cf. the introduction \S~\ref{sec:intro:cases} and Theorem~\ref{cor:moment}):
\begin{corollary}\label{cor:height} Let assumptions be as above and assume $D$ is prime. Then:
\begin{equation}\label{eq:cor:height}
\frac{\widehat{h}(\varphi(z_D))}{\deg(\varphi)}
=\frac{12}{N}\prod_{p|N}(1+\frac{1}{p})^{-1}
\biggl[\CmL_D + h_f+O_f(|D|^{-\eta_5})\biggr],
\quad
\text{as }
D\to -\infty,
\end{equation}
where:
\begin{equation}
 h_f:=\frac{L'}{L}(1,\MSym^2 f)-\frac{\zeta'}{\zeta}(2)-\gamma-\log 2\pi+\frac{1}{2}\log N+\sum_{p|N}\frac{p\log p}{p^2-1}.
\end{equation}
\end{corollary}
This result follows from \eqref{GZ}, \eqref{alpha}, Theorem~\ref{cor:moment} and the fact that all prime factors of $N$ are split in $\BmQ(\sqrt{D})$. It is consistent, except for a multiplicative constant\footnote{there is a discrepancy by a factor $2$ between the two results. The author has tried for a long time to settle the exact value of the constant. It is really difficult to do so in view of the number of distinct manipulations involved to establish Proposition~\ref{prop:geom} and Corollary~\ref{cor:height} and the Gross-Zagier formula~\cite{GZ}. Perhaps the $12$ should be $6$? We couldn't decide whether the mistake arises in the present article or in one of the formulas we quote from the literature}, with Proposition~\ref{prop:geom}.

\begin{remark} This asymptotic improves on a recent result by G.~Ricotta and T.~Vidick~\cite{RV05}*{Theorem 4.1}. Their result concerns the average of $\frac{\widehat{h}(\varphi(z_D))}{\deg(\varphi)}$ over $Y<D<2Y$, with $Y\to\infty$. The leading term is of the form (see also~\cite{RT08}):
\begin{equation}
 \log Y + h'_f +O(Y^{-\frac{1}{21}}).
\end{equation}
If we average~\eqref{eq:cor:height} we indeed recover that result because the average of $\CmL_D$ is proportional to $\log Y$ (one may also check that the average of $h_f$ agrees with $h'_f$).

More precisely our result uncovers the apparent complexity of~\cite{RV05}*{Figure 1} which plots the values of~\eqref{eq:cor:height} with $E$ an elliptic curve of conductor $37$ and $\abs{D}$ going up to $5.10^5$. The general trend is a logarithmic growth (which is consistent with the bound $\CmL_D \gg \log \abs{D}$) but, as the authors pointed out, the growth seems to be \Lquote{very irregular}. We may now explain this phenomenon by the fact that $\CmL_D-\Mdemi \log \abs{D}=\frac{L'}{L}(1,\chi_D)$ may take exceptionally large values (positive or negative), especially when the class number $h(D)$ is exceptionally small, which may happen in that range of discriminant. See~\cite{MM00}*{Figure 1} for a plot of $\frac{L'}{L}(1,\chi_D)$.
\end{remark}

\subsection{A challenging remark...} If one inspects the geometric approach of the previous section one may see that it is possible to prove:
\begin{equation}
 \sumchi
 \Lprime
 \ge C_f L(1,\chi_D)\log \abs{D},
 \qtext{for $D$ large enough,}
\end{equation}
without making use of any deep analytic estimate for quadratic $L$-series. Here $C_f>0$ depends\footnote{it is effective but the \Lquote{for $D$ large enough} is not.} only on $f$. Indeed we first make use of the Gross-Zagier formula~\eqref{GZ}, then the ingredients involved in the proof of Proposition~\ref{prop:geom} from section~\ref{sec:geom} consist of generalities on height functions plus the Chowla-Selberg formula. As recalled in \S~\ref{sec:intro:LD} the bound $\CmL_D\gg \log \abs{D}$ follows from Weyl's law on the zeros of $L(s,\chi_D)$.

\subsection{...and a reservation.}
However if we compare the situation to other $L$-functions associated to quadratic fields (Cases (A-D) discussed in the introduction), it is possible to make the previous observation slightly less surprising.

In Case (C), Waldspurger formula combined with the fact that cusp forms are bounded yields at once a $O(1)$ bound for the corresponding moment. But it is a consequence of Duke's equidistribution theorem that the moment has a positive limit as $D\to -\infty$ \cite{MV05}. 

In Case (A), a similar discussion occurs in~\cite{DFI95} which is even closer to our situation. The authors explain that the proof of \cite{DFI95}*{Theorem~2} is made \Lquote{using mostly elementary means} and still provide an asymptotic for the second moment -- this is to be compared with Proposition~\ref{prop:geom}. On the other hand the proof of \cite{DFI95}*{Theorem~3} demands \Lquote{a lot more work} and the use of Duke's theorem -- this is to be compared with Corollary~\ref{cor:height}.

\subsection{Appendix -- on multiplicative constants.}\label{sec:analytic:alpha} The determination of the value of $\alpha$ is quite puzzling since the normalizations in~\cite{GZ} are not always standard and are scattered through the text. Its exact value is important for us to check the consistency between section~\ref{sec:geom} and~\ref{sec:analytic}. In this paragraph we give some details. We hope this will be helpful to gain a better understanding of the underlying quantities.

Consider the diagram:
\begin{equation}
 \xymatrix{%
 X_0(N)\ar[dr]^\varphi \ar@{^{(}->}[r]^\iota & J_0(N) \ar[d] & E_f \ar[l] \ar[dl]\\
 & E &
 }
\end{equation}

The genuine Gross-Zagier formula, as it is proved in~\cite{GZ}*{Theorem 6.3 \S~I.6} or \cite{Zhan01b}*{Theorem 1.2.1} or \cite{YZZ08} is the identity:
\begin{equation}
 \frac{1}{h(D)}\Lprime=16\pi (f,f) L(1,\chi_D) \widehat{h}(\iota(z_D)_{f,\chi}),
\end{equation}
where in the right-hand side it is meant the $f,\chi$-isotypical component. It is possible to infer the equality $\dfrac{\widehat{h}(\varphi(z))}
{\Mdeg \varphi}
=
\widehat{h}(\iota(z)_f)$, see \cite{GZ}*{p.~310}. From the relation $(f,f)=\frac{N}{8\pi^3}L(1,\MSym^2 f)$ one deduces the value of $\alpha$ given in~\eqref{alpha}.

\section{Proof of Theorem~\ref{cor:moment}.}\label{sec:cor}
\subsection{Rankin-Selberg $L$-functions} The assumptions are as in Theorem~\ref{cor:moment}. From Rankin-Selberg theory we have a convolution representation of the $L$-function (see~\cite{GZ}*{Chap. IV (0.2)} for a proof of the following properties):
\begin{equation}\label{dirichlet}
 L(s,f\times \chi)=
 L^{(N)}(2s,\chi_D)
 \sum_{\Fma\subset \CmO_D} \chi(\Fma)\lambda_f(\TmN\Fma)\TmN\Fma^{-s}
 =:\sum^{\infty}_{n=1} \frac{a_n}{n^s},
 \text{ say.}
\end{equation}
The sum is over ideals of the ring of integers $\CmO_D$ of $\BmQ(\sqrt{D})$. We have an holomorphic continuation and if we set:
\begin{equation}
 \Lambda(s,f\times \chi)
 := \abs{ND}^s 
 \Gamma_\BmR(s+\Mdemi)
 \Gamma_\BmR(s+\frac{3}{2})
 L(s,f\times \chi),
\end{equation}
where $\Gamma_\BmR(s):=\pi^{-s/2}\Gamma(\frac{s}{2})$, the functional equation reads:
\begin{equation}\label{eq-fn}
 \Lambda(1-s,f\times \chi)=-\Lambda(s,f\times \chi),
 \quad
 \forall s\in \BmC.
\end{equation}

\subsection{Approximate functional equation.}\label{sec:cor:approximate} 
Recall that the Dirichlet $L$-series associated to principal automorphic representations are absolutely convergent for $\MRe s>1$ and the functional equation~\eqref{eq-fn} links it to $\MRe s<0$. The values lying in the \Lquote{critical strip} $0\le \MRe s\le 1$ own the deepest arithmetic glint on its coefficients $(a_n)_{n\ge 1}$. In~\eqref{approx-temp}, $L'(1/2,f\times\chi)$ is expressed as a weighted sum of the first $\abs{ND}$-coefficients, the so called \Lquote{approximate functional equation} method. This procedure is classical and we shall recall briefly what we need here, referring to~\cite{Harc02} or \cite{book:IK04}*{\S 5.2} for details.

Set $L_\infty(s):=\Gamma_\BmR(s+\Mdemi)
\Gamma_\BmR(s+\frac{3}{2})$, so in particular $L_\infty(\Mdemi)=\pi^{-1}$. Let us choose once and for all a meromorphic function $G$ such that:
\begin{itemize}
\itemb $G$ is holomorphic on $\BmC$ except at $0$, where we have:
\begin{equation}
  G(s)=\frac{1}{s^2} +O(1),\quad s\to 0,
\end{equation}
 \itemb $G$ is even: $G(s)=G(-s)\ \forall s\not=0$,
\itemb $G$ is of moderate growth (polynomial) on vertical lines.
\end{itemize}
(Actually one may simply choose $G(s):=1/s^2$ but there is no harm in retaining this degree of generality: mainly~\eqref{res-V} is needed in the sequel). Let $V\in \CmC^\infty$ be defined by:
\begin{equation}\label{mellin-V}
 V(y):=\int_{\MRe s=2} \widehat{V}(s) y^{-s}\frac{ds}{2i\pi},
 \quad
 y\in (0,\infty),
\end{equation}
where
\begin{equation}
 \widehat{V}(s):=\pi L_\infty(\frac{1}{2}+s)G(s),
 \quad s\in \BmC-\{0\}.
\end{equation}
It is not difficult to check that:
\begin{equation}\label{res-V}
 \widehat{V}(s)=\frac{1}{s^2}-2\frac{\gamma+\log 2\pi}{s}+O(1),
 \quad
 s\to 0.
\end{equation}

A standard contour argument shows, as consequence of~\eqref{eq-fn}, that:
\begin{equation}\label{approx-temp}
 L'(1/2,f\times \chi)=2 \sum^\infty_{n=1} \frac{a_n}{n^{1/2}}V(\frac{n}{\abs{ND}}).
\end{equation}
The sum is rapidly convergent and more precisely we have the following estimates.
\begin{lemma}\label{lem:estimate-V} For every integer $j\in \BmN$ we have:
\begin{equation}
 V^{(j)}(y)=
 \begin{cases}
  -(\log y)^{(j)} + O_j(y^{1/2-j})&
  \text{when $0<y\le 1$},\\
  O_{A,j}(y^{-A})&
  \text{when $1\le y$, for all $A>0$.}
 \end{cases}
\end{equation}
\end{lemma}
\begin{proof} When $0<y\le 1$, move the line of integration in~\eqref{mellin-V} to $\MRe s=-\frac{1}{2}$, crossing a pole at $s=0$ of residue $-\log y$ and estimate the remaining integral with Stirling formula. When $y\ge 1$, move the line of integration in~\eqref{mellin-V} to $\MRe s = A$. See also~\cite{book:IK04}*{Proposition 5.4}.
\end{proof}
From~\eqref{dirichlet}, the identity~\eqref{approx-temp} and by the orthogonality of characters on a finite abelian group we deduce the following formula that will be our starting point for the proof of Theorem~\ref{cor:moment}:
\begin{equation}\label{approx}
\sumchi \Lprime
=2 
\sum^\infty_{
\substack{m=1\\
(m,N)=1}}
\frac{\chi_D(m)}{m}
\sum^\infty_{n=1}
 \frac{r_D(n)\lambda_f(n)}{n^{1/2}} 
 V(\frac{m^2n}{\abs{ND}}).
\end{equation}
Here $r_D(n)$ is the number of elements of $\CmO_D$ of norm $n$, that is:
\begin{equation}\label{def-rD}
 2r_D(n):=\#
 \{ (a,b)\in \BmZ^2,
 \quad a^2-b^2D=4n
 \}
\end{equation}
(we have assumed $\abs{D}\ge 7$ odd).

\subsection{Main term}\label{sec:cor:main} Since $D$ appears several times in identity~\eqref{approx} it is not clear \emph{a priori} what is the main term as $D\to -\infty$. The aim of this paragraph is to give some explanations of how to riddle where it comes from.\footnote{see also~\cites{RV05,RT08} for a nearby discussion where the average over $D$ simplifies the situation.} 

Because of the weight $\frac{1}{m}$ the $m$-sum diverges gently enough (logarithmic growth) so that the sign $\chi_D(m)$ cannot really matter. The $n$-sum is really the key.

Let us consider the terms\footnote{One could view these as \emph{diagonal terms} by analogy with classical situations} that correspond to $b=0$ in the \Lquote{counting function} $r_D$. We shall show that these terms contribute a positive amount to the asymptotic (and it turns out that this will indeed constitute the main term of Theorem~\ref{cor:moment}). 

To have a feeling of this, one may view $r_D(n)$ as a \Lquote{probability density function} against which we sum the eigenvalues $\lambda_f$. When $1\le 4n <\abs{D}$, the density is located on the perfect squares, each of the same weight. Observe that this contribution comes from the $b=0$ terms only. This set is fixed and captures small Hecke eigenvalues of $\lambda_f$ so that it cannot cancel out (and has to contribute to the main term of the final asymptotic).
 
When $\abs{D}\le 4n < \abs{ND}^{1+\epsilon}$ one could say that the density is less sparse\footnote{this picture is not entirely truthful since we shall apply Theorem~\ref{th:main} which exhibits cancellations against the sparse sequence $n\mapsto n^2+d$} and we shall see in the next paragraph that when summing $\lambda_f$ against it, we indeed obtain cancellations (also observe that the weight $\frac{1}{n^{1/2}}$ diminishes the individual values of the summand).
\begin{lemma}\label{lemma:diagonal}
 The contribution in~\eqref{approx} from the terms $b=0$ is asymptotic to
 \begin{equation}\label{eq:lem:diagonal}
 \begin{aligned}
4\frac{L^{(N)}(1,\chi_D)}{\zeta^{(N)}(2)}L(1,\MSym^2
f)\biggl[\Mdemi\log 
&
|DN| +
  \frac{L'^{(N)}}{L^{(N)}}(1,\chi_D)+
\\
&+\frac{L'}{L}(1,\MSym^2 f)-\frac{\zeta'^{(N)}}{\zeta^{(N)}}(2)-\gamma-\log 2\pi +O_f(|D|^{-1/32+\epsilon})\biggr].
\end{aligned}
\end{equation}
\end{lemma}
\begin{proof}
 The contribution is equal to
\begin{equation}
 2 
\sum^\infty_{
\substack{m=1\\
(m,N)=1}}
\frac{\chi_D(m)}{m}
\sum^\infty_{a=1}
 \frac{\lambda_f(a^2)}{a} 
 V(\frac{a^2m^2}{\abs{ND}}).
\end{equation}
Recall that ($N$ is squarefree):
\begin{equation}
 L(s,\MSym^2 f)=\zeta^{(N)}(2s)\sum^\infty_{n=1}
 \frac{\lambda_f(n^2)}{n^s}
 ,\quad
 \text{for $\MRe s >1$.}
\end{equation}
From~\eqref{mellin-V} we deduce that the contribution is also equal to:
\begin{equation}
 2\int_{(2)} \frac{L^{(N)}(2s+1,\chi_D)}{\zeta^{(N)}(4s+2)}L(2s+1,\MSym^2 f) \widehat{V}(s)\abs{ND}^s \frac{ds}{2i \pi}. 
\end{equation}
We move the line of integration to $\MRe s=-\frac{1}{8}$, crossing a p\^ole at $s=0$. The residue is as given in~\eqref{eq:lem:diagonal}, as one may check from~\eqref{res-V}. The remaining integral is bounded thanks to the rapid decay of $\widehat{V}(s)$ as $\Mim s\to \pm \infty$ and Burgess subconvexity bound.
\end{proof}

\subsection{Remaining terms.}\label{sec:cor:remain} In view of the discussion in the previous paragraph, it is natural to introduce the function:
\begin{equation}\label{def-rD-dag}
 r^{\dagger}_D(n)
 :=\#
 \{ (a,b)\in \BmZ\times \BmN^\times,
 \quad a^2-b^2D=4n
 \}.
\end{equation}

From~\eqref{approx}, Lemma~\ref{lemma:diagonal}, and the forthcoming estimates it is easy to complete the proof of Theorem~\ref{cor:moment}. Observe that since $N$ is fixed the two ranges for the $m$ parameter in (i) and (ii) of the following proposition overlap to a large extent when $D\to-\infty$.

\begin{proposition}\label{prop:rem} We have the following uniform bounds. 

(i) When $1 \le m\le \abs{D}^{\eta_5}$:
\begin{equation}
  \sum^\infty_{n=1}
 \frac{r^{\dag}_D(n)\lambda_f(n)}
 {n^{1/2}}
 V(\frac{m^2n}{\abs{ND}})
 \ll_f \abs{D}^{-\eta_5}.
\end{equation}

(ii) When $2\sqrt{N}\le m<\infty$:
\begin{equation}
 \sum^\infty_{n=1}
 \frac{r^{\dag}_D(n)\lambda_f(n)}
 {n^{1/2}}
 V(\frac{m^2n}{\abs{ND}})
 \ll_{f,A,\epsilon} N^A\abs{D}^{1/2+\epsilon}m^{-2A},
 \qtext{for all $A,\epsilon>0$.}
\end{equation}
\end{proposition}
\begin{proof}
(ii) When $2\sqrt{N}\le m<\infty$, the estimate comes from the rapid decay of $V$. Indeed $r^{\dag}_D(n)>0$ only if $n\ge \abs{D}/4$, in which case $\frac{m^2n}{\abs{ND}}\ge \frac{m^2}{4N}\ge 1$. Therefore we may apply the second estimate in Lemma~\ref{lem:estimate-V}.

(i) Assume now that $1 \le m\le \abs{D}^{\eta_{5}}$ and write $4n=a^2-b^2D$. When $b\ge 2\sqrt{N}$, we have again $\frac{m^2n}{\abs{ND}}\ge 1$ and we apply Lemma~\ref{lem:estimate-V} as before. This yields a negligible contribution as soon as $b\ge \abs{D}^{\eta_5}$.
 
From now on we assume that $b\ge \abs{D}^{\eta_5}$. In a similar manner, we may assume up to a negligible term that $\abs{a}\le \abs{D}^{1/2+\eta_5}$. The contribution from $a=0$ is clearly negligible and thus it remains to estimate:
\begin{equation}
 \frac{m}{\abs{ND}^{1/2}}\sum_{1 \le a\le \abs{D}^{1/2+\eta_5}} \lambda_f(a^2-b^2D)W(\frac{(a^2-b^2D)m^2}{\abs{ND}})
\end{equation}
where $W(y):=V(y)y^{-1/2}$. Introduce:
\begin{equation}
 S_x:=\sum_{1 \le a \le x} \lambda_f(a^2-b^2D),
 \quad x\in\BmR_+,
\end{equation}
so that after integrating by parts we need to estimate:
\begin{equation}\label{proof:int-part}
 \frac{m^3}{\abs{ND}^{1/2}} \int^{\abs{D}^{1/2+\eta_5}}_1 S_x W'(\frac{(x^2-b^2D)m^2}{\abs{ND}}) \frac{x dx}{\abs{ND}}.
\end{equation}
We have $y:=\frac{(x^2-b^2D)m^2}{\abs{ND}} \ge \frac{1}{N}$, so that $W'(y)$ is bounded by $O(N^{3/2})=O_f(1)$.

We make use of Theorem~\ref{th:MAIN} to bound $S_x$. A straightforward dyadic subdivision yields:
\begin{equation}
 S_x \ll_f
 \abs{D}^{1/2-\eta}b^A \log \abs{D}
 ,\quad
 \forall x<\abs{D}^{1/2+\eta}.
\end{equation}
Inserting this bound in~\eqref{proof:int-part} yields
\begin{equation}
 \ll_f \abs{D}^{-\eta}\abs{D}^{3\eta_5}\abs{D}^{2\eta_5}\abs{D}^{\eta_5}
\end{equation}
which concludes the proof of the proposition.
\end{proof}

\section{On quadratic exponential sums.}\label{sec:sums}
As we shall see in the context of the proof of Theorem~\ref{th:main}, the following exponential sum arises naturally when one applies the $\delta$-symbol method (see section~\ref{sec:circle} and identity~\eqref{apply:delta}):
\begin{equation}\label{exp:intro}
  \frac{1}{q}\sum_{n\in \BmZ/q\BmZ}S(m,n^2+d;q)e_q(ln).
\end{equation}

This sum carries a \emph{square-root cancellation} in the sense that its typical size is $\tau(q)$ (as $q$ gets large). As explained in the introduction, this cancellation is not enough for our purpose and we shall need quantitative oscillations of the \Lquote{angle} (argument) as $q$ varies. In more concrete terms this means cancellations when summing over $q$ in an interval.

In this section we shall claim an estimate which is what we need to prove the main Theorem, see bound \eqref{eq:th:B} in Theorem~A. Ultimately the estimate would rely on Iwaniec's celebrated estimate for Fourier coefficients of half-integral forms~\cite{Iwan87}. We have decided not to include the proof of Theorem~A here because it is tedious and requires the introduction of a large number of objects. For these reasons and for the sake of clarity we postpone\footnote{We apologize to the reader if as a consequence the content of this section might appear a little mysterious at first sight.} the complete discussion and proof to the companion paper~\cite{Temp:quadratic}.

\remark It took a long time for the author to study and uncover the properties of the exponential sum~\eqref{exp:intro}. In the following we present the \emph{quickest way} to deal with it by recognizing a link with Jacobi forms. In the author's PhD thesis~\cite{Temp:these} we have established~\eqref{eq:th:B} under certain coprimality assumptions which would be enough for the proof of Theorems~\ref{th:main} and~\ref{th:MAIN}, see~\cite{Temp:cras} for an outline of a possible method via explicit evaluation of twisted Sali\'e sums and the equidistribution of roots of quadratic congruences~\cites{Hool63,DFI95}.

\subsection{A family of exponential sums.}\label{sec:sums:Jacobi} The following exponential sums appear in the Fourier expansion of Poincar\'e series for Jacobi forms, see~\cite{book:EZ}*{part I} (Eisenstein series) or \cite{GKZ}*{\S II.2} (general case).
\begin{definition}\label{def:J} For $q\ge 1$ and $n_1,n_2,r_1,r_2\in \BmZ$, let:
\begin{equation}
  J(n_1,r_1;n_2,r_2;q):=\frac{1}{q} e_{2q}(r_1r_2)
  \sum_{\substack{
  y\in \ZqZ\\
  x\in(\ZqZ)^\times
  }}
  e_q((y^2+r_1y+n_1)\overline{x}+n_2x+r_2y).
\end{equation}
\end{definition}

It is clear that we have the identity:
\begin{equation}\label{J-T}
 \frac{1}{q}\sum_{n\in \BmZ/q\BmZ}S(m,n^2+d;q)e_q(ln)=J(d,0;m,l;q).
\end{equation}

These exponential sums enjoy many properties, for instance the symmetry between the indices $1\leftrightarrow 2$. Here we recall the twisted multiplicativity property which is a straightforward consequence of the chinese remainder theorem.
For $q,q'\ge 1$ with $(q,q')=1$, one has:
\begin{equation}
 J(n_1,r_1;n_2,r_2;qq')=J(n_1\overline{q'}^2,r_1\overline{q'};n_2,r_2;q)J(n_1\overline{q}^2,r_1\overline{q};n_2,r_2;q').
\end{equation}

\subsection{Sums of exponential sums.}\label{sec:sums:subconvex}
This section contains the technical estimate that we shall need in the proof of Theorem~\ref{th:main}.
\begin{theoremA}\label{th:B}
Let $n_1,n_2,r_1,r_2\in\BmZ$ be such that $r^2_1-4n_1$ or $r^2_2-4n_2$ is non-zero. Put:
\begin{equation}
 C:=(\abs{r^2_1-4n_1}+1)(\abs{r^2_2-4n_2}+1).
\end{equation} 

(i) For all $\epsilon>0$,
\begin{equation}\label{eq:thA:squareroot}
 J(n_1,r_1,n_2,r_2;q)\ll_\epsilon (qC)^{\epsilon}.
\end{equation}

(ii) For $a\ge 1$ one has the following uniform estimate:
\begin{equation}\label{eq:th:B}
 \sum_{\substack{
 Q<q<2Q,\\
 q \equiv 0 \qmod{a}
 }} J(n_1,r_1,n_2,r_2;q)\ll Q^{1-\eta_1}a^A
\end{equation}
valid for all $Q$ with $1 \le Q<C^{1/2+\eta_1}$. Here $\eta_1,A>0$ are absolute constants. 
\end{theoremA}

A proof of these estimates is the main object of~\cite{Temp:quadratic}.

\subsection{A reduction.} 
The following lemma will allow the use of Theorem~A in the presence of a residual inverse $\overline{\CmN_2}$ such that $\CmN_2\overline{\CmN_2}\equiv 1\qmod{q}$. This occurrence will appear in the sequel, see equation~\eqref{apply:delta}.
 
\begin{lemma} Let $D$ be a fundamental discriminant, $l,m\in \BmZ$ and $e\ge 1$ be integers.
Let $\CmN_2\ge 1$ be odd squarefree and coprime with $D$ and $q$. Introduce $\CmN_2=\CmN_3\CmN_4$ with $\CmN_4|e$ and $(\CmN_3,e)=1$ and put $e':=e/\CmN_4$. We have the equality:
\begin{equation}\label{eq:lem:inverseN2}
 J(-e^2D,0,m\overline{\CmN_2},l;q)=
 \chi_D(\CmN_3)
 J(-e'^2D,0,m\CmN_2,l\CmN_2;q\CmN_3).
\end{equation}
\end{lemma}
\remark It is important to observe that the identity~\eqref{eq:lem:inverseN2} (nor any naive variant) is not true in general without the coprimality assumptions. We see this clearly in the proof where a multiplicative factor \emph{has to be non-zero}. This is the main obstruction why we have made the restrictions on the level $\CmN$ and the primality of $D$ in Theorem~\ref{th:main}.
\begin{proof}
Since $q$ is coprime with $\CmN_2$ we have by twisted multiplicativity:
\begin{equation}
 J(-e'^2D,0,m\CmN_2,l\CmN_2;q\CmN_3)=
 J(-e'^2D\overline{\CmN^2_3},0,m\CmN_2,l\CmN_2;q)
 J(-e'^2D\overline{q^2},0,m\CmN_2,l\CmN_2;\CmN_3).
\end{equation}
The first term of the right-hand side is equal to $J(-e^2D,0,m\overline{\CmN_2},l;q)$ by change of variable $(x,y)\leadsto (\CmN^2_2 x,\overline{\CmN_2}y)$ in the Definition~\ref{def:J} of $J$.

The second term of the right-hand side is equal to $\chi_D(\CmN_3)$ which concludes the proof of the lemma since $\chi_D(\CmN_3)$ is non-zero. Indeed it is not difficult to see that this term is equal to ($\CmN_3$ is squarefree coprime with $e'$):
\begin{equation}
 J(-D,0,0,0;\CmN_3)=\prod_{p|\CmN_3} 
 \sum_{\substack{
  n\in \ZpZ\\
  x\in(\ZpZ)^\times
  }}
  e_p((n^2-D)x)
\end{equation}
Expliciting the Ramanujan sum ($x$-variable) the last sum is:
\begin{equation}
 \#\{n(p); n^2\equiv D(p)\}-1=\Legendre{p}{D}=\chi_D(p).
\end{equation}
The last equalities hold because $p|\CmN_3$ is odd.
\end{proof}

\section{A variation on the \texorpdfstring{$\delta$}{delta}-symbol method.}\label{sec:circle}
Despite the apparent routine of this section, the estimates are really delicate. The $q$-variable is particularly sensitive: for instance the $(q\Omega+\abs{u})^{-1}$ from Lemma~\ref{lem:estimateDeltaq} cannot be replaced by $q^{-1}\Omega^{-1}$ without damaging the proof in the next section.

\subsection{Vorono\"i summation formula.} In detecting cancellations in sums of Fourier coefficients the Vorono\"i summation formula is a convenient and classical tool. We shall use the following variant, borrowed from \cite{KMV02}*{Theorem~A.4}:
\begin{proposition} Let $f$ be a primitive new form of weight $2$ and level $\CmN$. Assume that $q$ is such that $(q,\frac{\CmN}{q})=1$. Let $\CmN_1$ and $\CmN_2\ge 1$ be such that
\begin{equation}\label{decompositionN1N2}
  \CmN=\CmN_1\CmN_2~;~\CmN_1=(q,\CmN)~;~(q,\CmN_2)=1.
\end{equation}
Let $d$ be an integer prime with $q$ and $g$ be a smooth function of compact support. Then:
\begin{equation}
 \sum^\infty_{m=1} \lambda_f(m)e(\frac{md}{q})g(n)=-2\pi\frac{\eta_f(\CmN_2)}{q\sqrt{\CmN_2}}
\sum^\infty_{m=1}\lambda_f(m)e(-m\frac{\overline{d\CmN_2}}{q})\widetilde{g}(m;q)
\end{equation}
where $\eta_f(\CmN_2)$ is a complex number of modulus $1$ and:
\begin{equation}
 \widetilde{g}(y;q):=\int^\infty_0 g(y)J_1(\frac{4\pi\sqrt{xy}}{q\sqrt{\CmN_2}})dy.
\end{equation}
\end{proposition}

\remark When $\CmN$ is squarefree, the condition on $q$ is always fulfilled. In the sequel we do place ourselves in this case and shall use the decomposition $\CmN=\CmN_1\CmN_2$ from~\eqref{decompositionN1N2} without further indication (but one should be aware that $\CmN_1$ and $\CmN_2$ depend on $q$, or more precisely on $(q,\CmN)$).

\subsection{Setting-up the \protect{$\delta$}-symbol.} The capital letters $U,\Omega,Q$ shall denote the length of various sums. We postpone the definitive choice of these quantities until the next \S~\ref{subsec:choice} for the sake of clarity. These choices are slightly unusual. The principles underlying this section are known, and we shall follow~\cite{DFI2} closely. A difference is that we shall need here a control on the \emph{smoothness in the $q$-variable}. This aspect is crucial for our purpose (recall that the cancellations ultimately come from the $q$-sum, cf. Theorem~A), therefore we include brief proofs of the key estimates.

Fix once and for all a function $W\in \CmC^\infty_c((-1,1))$ with $W(0)=1$ and put $\phi(u):= W(\dfrac{u}{U})$, $u\in \BmR$. Let $\omega$ be a smooth function of compact support in $(\Omega,2\Omega)$ such that (the constants are absolute):
\begin{equation}
  \sum^\infty_{r=1}\omega(r)=1~;~\omega^{(i)}\ll_i \Omega^{-i-1},\ \forall i\in \BmN.
\end{equation}
The $\delta$-symbol, which is $1$ if $n=0$ and $0$ else, is expressed with additive characters (Ramanujan sums):
\begin{equation}\label{def:delta}
\delta(n)=\phi(n)
\sum_{q\ge 1} 
\Delta_q(n) \sideset{}{^*}\sum_{d(q)} e(\frac{nd}{q}),
\quad \forall n\in\BmZ
\end{equation}
where
\begin{equation}
 \Delta_q(u):=\sum^\infty_{r=1}\frac{1}{qr}\left[\omega(qr)-\omega(\frac{u}{qr})\right],
 \qtext{for $u\in\BmR$.}
\end{equation}

\begin{lemma}\label{lem:estimateDeltaq} (i) The function $\Delta_q\phi$ identically vanishes unless $1 \le q\le Q$, where
$Q:=\max(\Omega,\dfrac{U}{\Omega})$.
 
 (ii) The high derivatives of $\Delta_q$ satisfy, for all $i>0$:
\begin{equation}\label{eq:lem:Deltaq}
 \Delta_q^{(i)}(u) \ll_i (q\Omega)^{-i-1},\quad u\in\BmR.
\end{equation}
 
(iii) We have the following uniform bounds ($u\in\BmR$):
\begin{equation}\label{estimateDeltaq}
\Delta_q(u)\ll \Omega^{-2} + (q\Omega+\abs{u})^{-1}.
\end{equation}

(iv) We have the following bound for the derivative:
\begin{equation}\label{estimate-Deltaq-deriv}
\frac{\partial}{\partial q}\Delta_q(u)\ll q^{-1}\Omega^{-2} + q^{-2}\Omega^{-1}.
\end{equation}
\end{lemma}

\remark The bounds in~\eqref{eq:lem:Deltaq} and~\eqref{estimate-Deltaq-deriv} are uniform in the $u$-variable, which is sufficient for our purpose. On the other hand the presence of $u$ as $(q\Omega+\abs{u})^{-1}$ in~\eqref{estimateDeltaq} is necessary in the sequel.

\begin{proof} Claims (i) and (ii) are immediate. The proof of claim (iii) may be found in~\cite{DFI2}*{Lemma~2}. We repeat it for convenience. We use of the inequality $\{\frac{r}{q}\}\le \min(1,\frac{r}{q})$. Observe that
\begin{equation}
  \int^\infty_0 (\omega(r)-\omega(\dfrac{u}{r}))\dfrac{dr}{r}=0.
\end{equation}
 This implies (Euler-Maclaurin formula of order $1$):
   \begin{equation}\begin{split}
     \Delta_q(u)=\int^\infty_0 \{\frac{r}{q}\} d\frac{\omega(r)-\omega(\frac{u}{r})}{r} &\ll \int^\infty_0 |d\frac{\omega(r)}{r}| + \min(1,\frac{r}{q})|d\frac{\omega(\frac{u}{r})}{r}| \\
     &\ll \Omega^{-2}+\min(q^{-1}\Omega^{-1},\abs{u}^{-1}).
   \end{split}\end{equation}

The proof of (iv) is similar, write:
\begin{equation}\begin{split}
  \frac{\partial}{\partial q}\Delta_q(u)&=\sum^\infty_{r=1} \frac{1}{q^2r}[\omega(\frac{u}{qr})-\omega(qr)]+\frac{1}{q}\omega'(qr)+\frac{u}{q^3r^2}\omega'(\frac{u}{qr})\\
  &=\int^\infty_0 \{\frac{r}{q}\}d\bigl(\frac{1}{qr}\omega(\frac{u}{r})-\frac{1}{qr}\omega(r)+\frac{1}{q}\omega'(r)+\frac{u}{qr^2}\omega'(\frac{u}{r})\bigr)\\
  &\ll \int^\infty_0 \frac{r}{q}|d\frac{1}{qr}\omega(\frac{u}{r})|+|d\frac{1}{qr}\omega(r)|+|d\frac{1}{q}\omega'(r)|+\frac{r}{q}|d\frac{u}{qr^2}\omega'(\frac{u}{r})|\\
  &\ll q^{-2}\Omega^{-1}+q^{-1}\Omega^{-2}+q^{-1}\Omega^{-2}+q^{-2}\Omega^{-1}. \qedhere
\end{split}\end{equation}
\end{proof}

Let $h\ge 1$ be an integer and assume from now that $U$ is chosen such that: 
\begin{equation}\label{circle:Uh}
 U\le h/2.
\end{equation}
In particular $m\mapsto \Delta_q\phi(m-h)$ vanishes unless $m\ge 1$. Writing $\lambda_f(h)=\sum_m \lambda_f(m)\delta(h-m)$ and inserting the expression~\eqref{def:delta} for the $\delta$-symbol yields:
\begin{equation}
\lambda_f(h)
=\sum_{q\ge 1}\sideset{}{^*}\sum\limits_{d (q)}e(-d\frac{h}{q})
\sum_{m\ge 1}
\lambda_f(m)\Delta_q\phi(m-h)e(\frac{dm}{q}).
\end{equation}
We may apply Vorono\"i summation formula to the $m$-sum because the function
\begin{equation}
 g(x;h;q):=\Delta_q\phi(x-m)
\end{equation}
is of compact support thanks to the function $\phi$. This gives an expansion of $\lambda_f$ in terms of sums of Kloosterman sums:
\begin{proposition}\label{prop:circle} Under condition~\eqref{circle:Uh}, we have:
\begin{align}\label{ccl:circle}
\lambda_f(h)
&=-2\pi\sum_{m\ge 1}
\lambda_f(m)
\sum_{q\ge 1}
\frac{\eta_f(\CmN_2)}{q\sqrt{\CmN_2}}
S(m\overline{\CmN_2},h;q)\widetilde{g}(m;h;q),\\*
\intertext{where $S(\cdot,\cdot;q)$ denotes the classical Kloosterman sum and:}\label{g-tildeg}
\widetilde{g}(y;h;q)&:=\int^\infty_0{g(x;h;q)J_1(\frac{4\pi\sqrt{xy}}{q\sqrt{\CmN_2}})dx}.
\end{align}
\end{proposition}
\begin{remark}
If the weight of $f$ were $\ge 4$ (as in~\cite{Blom08}), we could have used the fact that Poincar\'e series span the finite dimensional space of holomorphic forms of level $\CmN$; and combine this with their explicit Fourier expansion (which is close to the right-hand side of~\eqref{ccl:circle}). This approach doesn't work for Maass forms and weight $2$ forms. Also the $\delta$-symbol offers more flexibility in the choice of the test function: here the function $\widetilde{g}$ shall decay rapidly as $y\to\infty$ and vanishes unless $q\le Q$. See \cite{Pitt}*{Introduction} for a similar discussion.
\end{remark}

The proof of the following lemma is straightforward (we make use of the fact that $q\le Q \le U/\Omega$):

\begin{lemma} (i) Unless $x\in (h-U,h+U)$, $g(x;n;q)$ vanishes.

(ii) The high derivatives of the function $x\mapsto g(x;n;q)$ satisfy ($i\in\BmN$):
\begin{equation}\label{G}
g^{(i)}(x;h;q) \ll_i C_1\times \min(U,q\Omega)^{-i}.
\end{equation}
The factor $C_1=C_1(|D|,U,\Omega)$ is a polynomial in $|D|,U$ and $\Omega$ whose coefficients do not depend on $i$.
\end{lemma}
 The following estimate is classical but we shall provide a quick proof because of its importance.
\begin{lemma} The Haenkel transform $\widetilde{g}$ satisfies, for any integer $A>0$:
\begin{equation}\label{haenkeldecay}
\widetilde{g}(y;h;q) \ll_A  C_2 \times \Bigl(\frac{y}{h}\Bigr)^{-A} \min(U/q,\Omega)^{-2A}.
\end{equation}
Here $C_2=C_2(y,|D|,U,\Omega)$ is a polynomial in $y,|D|,U$ and $\Omega$.
\end{lemma}
\begin{proof}
The basic idea is to integrate by part the Bessel function. An elegant way is to use
the following formula, see \cite{thes:Harcos}*{p.~51} or \cite{HM06}:
\begin{equation}
J_1(\sqrt{z})=\sum\limits^{2A}_{a=0}c_{a,A} z^{a-A} \left[J_{1+a}(\sqrt{z})\right]^{(a)}.
\end{equation}
The constants $c_{a,A}$ are absolute ; in the following $a$ or $\alpha$ denote an arbitrary integer between $0$ and $2A$.
\begin{align*}
\int^\infty_0{g(x)J_1(\frac{4\pi\sqrt{xy}}{q})dx} &\ll
\int^\infty_0{g(\frac{q^2z}{y})J_1(\sqrt{z})dz}
\ll_A 
\sum\nolimits_{a}
\abs{%
\int^\infty_0
\Bigl[g(\frac{q^2z}{y})z^{a-A}\Bigr]^{(a)}
J_{1+a}(\sqrt{z})dz
}
\\
&\ll_A
\sum\nolimits_a
\sum\nolimits_\alpha
\abs{%
\int^\infty_0
\Bigl[g(\frac{q^2z}{y})\Bigr]^{(\alpha)}
z^{\alpha-A}J_{1+a}(\sqrt{z})dz
}
\\
&\ll_A 
C_2\times
\sum\nolimits_\alpha
\min(U,q\Omega)^{-\alpha} 
\Bigl(\frac{q^2}{y}\Bigr)^\alpha
\Bigl(\frac{hy}{q^2}\Bigr)^{\alpha-A}\\
&\ll_A
C_2\times
\Bigl(\frac{yh}{q^2}\Bigr)^{-A}
\sum\nolimits_\alpha
\Bigl(\frac{h}{\min(U,q\Omega)}\Bigr)^{\alpha}.
\end{align*}
In the second line we have used the fact that the support of $g$ is included in $(h-U,h+U)$. The claim follows because $U < h$ so that $\alpha=2A$ is the dominant term in the last sum.
\end{proof}

\subsection{Restriction.}\label{subsec:restriction} From~\eqref{haenkeldecay}, the function $\widetilde{g}$ is very small when $\dfrac{y}{h}\min(U/q,\Omega)^2>|D|^{\eta_4}$. Thus, up to a negligible term, we may restrict the $m$-summation in equation~\eqref{ccl:circle} to (we use the fact that $q\le Q$):
\begin{equation}\label{restriction1}
1 \le m \le \frac{h}{U}\times \max(\frac{\Omega^2}{U},\frac{U}{\Omega^2}) \times |D|^{\eta_4}.
\end{equation}

\remark Because of~\eqref{circle:Uh}, the right-hand side is always greater than $1$. This is consistent with~\eqref{ccl:circle} in the sense that the sum of the RHS certainly cannot be void whatever the choice of $U$ and $\Omega$!

\subsection{Choice of the parameters.}\label{subsec:choice}
We make now explicit the choice of the initial parameters $U$ and $\Omega$. Later on, the integer $h$ will be such that $|D|^{1-\eta_4} < h < |D|^{1+\eta_4}$.
We choose:
\begin{equation}\label{circle:choice}
U:= |D|^{1-2\eta_4}~;~
\Omega := |D|^{1/2-\eta_4}.
\end{equation}
As a consequence, $Q = |D|^{1/2-\eta_4}$; and inequality~\eqref{restriction1} becomes:
\begin{equation}\label{restriction}
 1\le m \le |D|^{3\eta_4}.
\end{equation}

\section{Proof of Theorem~\ref{th:main}.}\label{sec:proof}
In this section, we establish Theorem~\ref{th:main}, making use of results from sections~\ref{sec:sums} and \ref{sec:circle}. First we state a more general version which was needed in the application to moments of quadratic $L$-functions (Theorem~\ref{cor:moment}, section~\ref{sec:cor}):
\begin{theorem}\label{th:MAIN} Let $\epsilon>0$.
There exist an absolute constant $A$ and a real number $\eta=\eta(\epsilon)>0$ depending on $\epsilon$ only such that the following holds.
Let $f$ be a modular form of weight $2$ and odd squarefree level and denote by $\lambda_f$ its normalized Fourier coefficients, see \eqref{intro:Fourier}. Then:
\begin{equation}\label{eq:th:MAIN}
\sum_{N<n<2N}\lambda_f(n^2-De^2)\ll_f |D|^{1/2-\eta} e^A,
\end{equation}
for all triples $(D,e,N)$ where $D$ is a fundamental negative discriminant whose prime factors are all greater than $|D|^\epsilon$ and $e$ and $N$ are positive integers with $N< |D|^{1/2+\eta}$. The implied constant depends on $f$ only (a polynomial in its level).
\end{theorem}

\remark Theorem~\ref{th:main} (where $d:=-De^2$) corresponds to the particular case where $e=1$ and $D$ is a prime discriminant. In that case we may choose $\epsilon=1$ so that $\eta>0$ is an absolute constant, as claimed. Actually we expect Theorem~\ref{th:MAIN} to hold with an absolute $\eta$ and without the constraint on the prime factors of $D$.

\subsection{Reduction to a smooth version.}\label{sec:proof:smooth} First we consider the equivalent smooth version of~\eqref{eq:th:MAIN} (see for instance~\cite{DFI95}*{\S~4} or \cite{book:IK04}*{\S~5.6} for some details on how to \Lquote{smooth things out}). We ought to prove that there exist absolute constants $A>0$ and $s\in \BmN$ as well as a real number $\eta_4=\eta_4(\epsilon)$, depending on $\epsilon$ only such that
\begin{equation}\label{temp:smooth}
  \sum^\infty_{n=1} \lambda_f(n^2+d) V(\frac{n}{N}) 
  \ll_f |D|^{1/2-\eta_4} 
  \cdot e^A\cdot
  \max_{0\le i\le s} ||V^{(i)}||
\end{equation}
holds uniformly, where:
\begin{itemize}
 \item[$\bullet$] $V\in \CmC_c^\infty((1,2))$;
 \item[$\bullet$] $D$ is a negative discriminant whose prime factors greater than $|D|^\epsilon$;
 \item[$\bullet$] $N>0$ is a real number such that $N<|D|^{1/2+\eta_4}$;
 \item[$\bullet$] $e\ge 1$ is an integer.
\end{itemize}

Here $\pnorm{.}=\pnorm{.}_\infty$ denotes the sup norm. The estimate~\eqref{temp:smooth} is trivial unless $e$ is a very small power of $\abs{D}$. It is also trivial when $N$ is much smaller than $D$. Thus we may and do assume in the sequel that
\begin{equation}
  |D|^{1-\eta_4} < N < |D|^{1+\eta_4} \text{ and } 1 \le e < |D|^{\eta_4}.
\end{equation}

\begin{proof}[Proof that \eqref{temp:smooth} implies \eqref{eq:th:MAIN}] We choose $V\in\CmC^\infty_c((1,2))$ which is $1$ on the interval $(1+\delta,2-\delta)$ and such that $V^{(i)}\ll_i \delta^{-i}$ (absolute constants). Then:
\begin{align*}
\sum_{N<n<2N} \lambda_f(n^2-De^2)
&\ll 
\abs{\sum^\infty_{n=1} \lambda_f(n^2-De^2) V(\frac{n}{N})}
+ \delta N\cdot \log N \\
&\ll_f |D|^{1/2-\eta_4}\cdot \delta^{-s}+\delta N\cdot \log N\\
\intertext{We choose $\delta=|D|^{-(\eta+\eta_4)/(1+s)}$, and put $\eta:=\eta_4/(1+2s)$:}
&\ll_f |D|^{1/2-\eta}\log |D|.
\end{align*}
In the first line we have made use of Deligne's bound: $\abs{\lambda_f(n)}\le \tau(n)$ for all $n\in\BmN^\times$. In the second line we have made use of assumption~\eqref{temp:smooth}.
\end{proof}

\subsection{Applying the \protect{$\delta$}-symbol method.}\label{sec:proof:delta} From now on, denote by $S$ the left-hand side of~\eqref{temp:smooth}. To ease notations, put $d:=-De^2$. By Proposition~\ref{prop:circle} we have:
\begin{equation}\label{apply:delta}
 S\ll
 \sum^\infty_{m=1}
\abs{%
\lambda_f(m)}
 \abs{%
\sum_{q\ge 1}
\frac{1}{q}
\sum^{\infty}_{n=1}
S(m\overline{\CmN_2},n^2+d;q) 
V(\frac{n}{N})
\widetilde{g}(m;n^2+d;q)
}
\end{equation}
The assumption $|D|^{1-\eta_4}<h<|D|^{1+\eta_4}$ from \S~\ref{subsec:restriction} is satisfied because $h=n^2+d$ and $N<n<2N$.
From~\eqref{restriction} we may and do cut the sum $S=S_1+S_2$ into two pieces. In $S_2$ we restrict the summation to $m \le |D|^{3\eta_4}$ up to a negligible error term:
\begin{equation}
 S_1\ll_A \abs{D}^{-A\eta_4},
 \qtext{for all $A>0$.}
\end{equation}

\subsection{Applying Poisson formula.}\label{sec:proof:poisson} Cancellations in~\eqref{apply:delta} arise both from the $q$ and $n$ sums. First we apply Poisson summation formula  to the $n$-sum (the outcome is -- roughly speaking -- that the $n$-sum occupies the (mod $q$) residue classes uniformly).
\footnote{A puzzling remark is the following. We have explained how to smooth the sum from~\eqref{eq:th:MAIN} to~\eqref{temp:smooth}. This smoothness is necessary to apply Poisson formula to~\eqref{apply:delta}. If identity~\eqref{apply:delta} were in its unsmooth form (i.e. $N<n<2N$) it would not be possible to smooth it out because inserting the Weil's bound for Kloosterman sums in~\eqref{apply:delta} would yield a bound much worse that $\ll \abs{D}^{1/2+\epsilon}$ (in fact $\abs{D}^{3/4+\epsilon}$). This is because we really need cancellations in both the $q$ and $n$ sums. In other words it is not possible to reverse the order of transformations. First smoothing (\S \ref{sec:proof:smooth}) and then applying $\delta$-symbol (\S \ref{sec:proof:delta}) is the \emph{sole sequence}.} Recall that $\widetilde{g}$ is zero unless $q\le Q=|D|^{1-\eta_4}$ and that $\CmN_2=\CmN/(\CmN,q)$ depends (mildly) on $q$. 

\begin{lemma}\label{lem:poisson} For each $m,q\ge 1$, we have:
\begin{equation}
\sum\limits^\infty_{n=1} S(m\overline{\CmN_2},n^2+d;q)
\widetilde{g}(m;n^2+d;q)V(\frac{n}{N}) = \sum_{l\in\BmZ} h(m;l;q)\sum_{n\in \BmZ/q\BmZ}S(m\overline{\CmN_2},n^2+d;q)e(\frac{ln}{q}),
\end{equation}
where $h(m;l;q)$ is defined below by~\eqref{def:h} and satisfies:
\begin{equation}\label{prop:h-l}
  h(m;l;q)\ll_A C_3 \times (\frac{|D|^{1/2}l}{q})^{-A} \times |D|^{A\eta_3}
  \max_{0\le i\le A}
  \pnorm{V^{(i)}}
  ,\qtext{for all $A>0$, $l\not=0$.}
\end{equation}
Here $C_3=C_3(m,l,q,|D|)$ is a polynomial in $m,l,q$ and $|D|$. Furthermore:
\begin{equation}\label{prop:h-0}
  h(l;m;q)=0 \qtext{unless $1\le q\le Q$},
\end{equation}
and we have the uniform bound:
\begin{equation}\label{eq:prop:h}
 h(m,l;q)\ll q^{-1} |D|^{1/2+\eta_3}\pnorm{V}
\end{equation}

The first derivative satisfies:
\begin{equation}\label{prop:h-q}
  \frac{\partial}{\partial q} h(m;l;q)\ll q^{-3}|D|^{1+\eta_3}\pnorm{V'},
  \qtext{for all $l\in \BmZ$ and $1 \le q\le Q$.}
\end{equation}
\end{lemma}

\begin{proof}
 By Poisson summation formula we have:
\begin{equation}
 \sum^{\infty}_{t=-\infty}\widetilde{g}(m;(n+tq)^2+d;q)
 =\sum_{l\in\BmZ} 
 e(\frac{ln}{q}) h(m;l;q)
\end{equation}
where
\begin{equation}
h(m;l;q):=\frac{1}{q}\int_{-\infty}^{\infty}
\widetilde{g}(m;z^2+d;q)V(\frac{z}{N})e(-\frac{lz}{q})dz.
\end{equation}
It is clear that~\eqref{prop:h-0} holds. 

The estimate~\eqref{prop:h-l} follows by repeated integration by parts once we know that
\begin{equation}\label{proof:gtilde}
 \frac{\partial^i}{\partial z^i}\widetilde{g}(m;z^2+d;q) \ll_i C_3 \times (\frac{U}{N})^{-i}.
\end{equation}
(because of $\frac{U}{N}>|D|^{1/2-\eta_4}$). Estimate~\eqref{proof:gtilde} follows from the corresponding estimate for $g$ and formula~\eqref{g-tildeg}. One needs to differentiate $\Delta_q\phi(m-d-z^2)$ in the $z$-variable, and for this it is enough to observe that $\phi^{(i)}\ll_i U^{-i}$ and $\Delta^{(i)}_q\ll_i (\dfrac{\Omega qr}{N})^{-i} \ll_i (\dfrac{u}{N})^{-i}\ll_i (\dfrac{U}{N})^{-i}$ for all $i\in \BmN$.

Consider now estimate~\eqref{prop:h-q}. Inserting the formula~\eqref{g-tildeg} for $\widetilde{g}$ in the definition of $h$ yields:
\begin{equation}\label{def:h}
  h(l;m;q)=\frac{1}{q}\int^\infty_{-\infty}\int^{+\infty}_0 \Delta_q\phi(x+b^2D-z)V(\frac{z}{N})J_1(\frac{4\pi\sqrt{xm}}{q\sqrt{\CmN_2}})e(-\frac{lz}{q}) dxdz
\end{equation}
The Bessel function satisfies (rough bound, absolute constants):
\begin{equation}
  J_1(z)\ll (1+z)^{-1/2}\ll 1~;~J'_1(z)=\Mdemi(J_0(z)-J_2(z))\ll 1.
\end{equation}
A bound for $\Delta_q$ and its $q$-derivative is given in~\eqref{estimateDeltaq}. From:
\begin{equation}
 \int_{x\sim M} (q\Omega+\abs{x})^{-1}dx \ll \log(q\Omega+M),
\end{equation}
we deduce that $h(l;m;q)$ is bounded by:
\begin{equation}
  q^{-1}(\Omega^{-2}M+\log(q\Omega+M))\times N \ll q^{-1}\Omega^{-2}N\max(M,\Omega^2) \ll q^{-1}|D|^{1/2+\eta_3}.
\end{equation}
When introducing the differentiation $\frac{\partial}{\partial q}$, we obtain a sum of four terms of the same kind and the previous bound get multiplied by
\begin{equation}
  q^{-1}+q^{-1}+q^{-2}\sqrt{Mm}+q^{-2}lN\ll q^{-2}|D|^{1/2+\eta_3}
\end{equation}
which yields~\eqref{prop:h-q}.
\end{proof}

\remark We have seen in the proof that:
\begin{equation}
  h(m;l;q)\ll q^{-2}|D|^{1+\eta_3}\pnorm{V}.
\end{equation}
(this also follows from~\eqref{prop:h-q} and~\eqref{prop:h-0} or might be checked directly from $\widetilde{g}(l;m;q) \ll q^{-1}|D|^{1/2+\eta_3}$). This bound is of the same strength as~\eqref{eq:prop:h} \emph{as long as} $q$ is near $\abs{D}^{1/2}$ (where particular the functions $h$ and $\widetilde{g}$ are bounded by an arbitrary small power of $|D|$). However estimate~\eqref{eq:prop:h} is necessary to tail the $q$-sum for small $q$'s. This observation is usefull to keep track of the estimates during the proof of Theorem~\ref{th:MAIN}.

\subsection{End of the proof.} From~\eqref{apply:delta} and Lemma~\ref{lem:poisson} it remains to estimate:
\begin{multline}\label{summary}
 S_2=\sum_{1 \le m \le |D|^{3\eta_4}} 
 \abs{\lambda_f(m)} 
 \sum_{l\in \BmZ} 
 \sum_{
 \substack{
 \CmN=\CmN_1\CmN_2\\
 (\CmN_1,\CmN_2)=1}}
 \abs{%
  \sum_{\substack{(q,\CmN_2)=1\\ \CmN_1|q}}
 h(m;l;q) \frac{1}{q} \sum_{n\in \BmZ/q\BmZ}S(m\overline{\CmN_2},n^2+d;q)e(\frac{ln}{q})
 }.
\end{multline}
Recall that the (complete) exponential sum has square-root cancellation, see~\eqref{eq:thA:squareroot} from Theorem~A. To conclude the proof we appeal to cancellations in the $q$-sum.

Up to a negligible error term we may restrict the $l$-sum to $\abs{l} < |D|^{2\eta_3}$ -- this is because of~\eqref{prop:h-l}. The quantity in absolute values is (see Definition~\ref{def:J} and relation \eqref{J-T}):
\begin{equation}
 E:=\sum_{\substack{q,\ \CmN_1|q \\(q,\CmN_2)=1}} 
 h(m;l;q) 
 J(d,0,m\overline{\CmN_2},l;q)=E_1+E_2,
\end{equation}
where $E_1$ contains the terms with $q\le \abs{D}^{1/2-2\eta_2}$. Making use of~\eqref{eq:prop:h} and~\eqref{eq:thA:squareroot} one has:
\begin{equation}
 E_1 \ll \abs{D}^{1/2-\eta_2}\pnorm{V}.
\end{equation}

For the remaining terms we perform an integration by parts (we use the fact that $h(m;l;q)$ is zero unless $q\le Q$), and utilize~\eqref{eq:lem:inverseN2}:
\begin{equation}
E_2= 
 -\chi_D(\CmN_3)
 \int^Q_{\abs{D}^{1/2-2\eta_2}}
 \frac{\partial}{\partial x} h(m;l;x) 
 \Biggl\{
 \sum_{\substack{1\le q\le x\\(q,\CmN_2)=1 ; \CmN_1\CmN_3|q}} J(-e'^2D,0,m\CmN_2,l\CmN_2;q)
 \Biggr\}
 dx.
\end{equation}
From~\eqref{prop:h-q}, \eqref{J-T}, and estimate~\eqref{eq:th:B} from Theorem~A we deduce:
\begin{equation}\begin{aligned}\label{lastbound}
E_2&\ll
\int^Q_{|D|^{\frac{1}{2}-2\eta_2}}
x^{-3}|D|^{1+\eta_3}
\abs{D}^{1/2-\eta_1}
(elm\CmN)^A
\pnorm{V'}
dx
\\&\ll
|D|^{1/2-\eta_2}(elm)^A \pnorm{V'}.
\end{aligned}\end{equation}

Returning to~\eqref{summary}, we bound trivially the sums on $\CmN_1,m$ and $l$. This yields 
\begin{equation}
  S_2\ll \abs{D}^{1/2-\eta_4}e^A(\pnorm{V}+\pnorm{V'})
\end{equation}
and concludes the majoration of~\eqref{temp:smooth} and the proof of Theorem~\ref{th:main}.




\end{document}